\documentclass[]{imsart}
\usepackage{amsmath,amsfonts,amssymb,amsthm,graphicx}
\usepackage{dsfont}    % it's a bitmap font (on Volodya's computer)
\usepackage{enumerate}
\usepackage{color}
\usepackage[numbers]{natbib}
\usepackage{multirow}

\usepackage{subcaption}
\usepackage{algorithm}
\usepackage[noend]{algpseudocode}  % algpseudocode loads algorithmicx
\usepackage[colorlinks=true,citecolor=blue,pdfpagemode=UseNone,pdfstartview=FitH]{hyperref}

\renewcommand{\d}{\,\mathrm{d}}

\newcommand{\p}{\mathbb{P}}

\newcommand{\var}{\mathrm{var}}   % variance

\newcommand{\E}{\mathbb{E}}    % expectation
\newcommand{\R}{\mathbb{R}}    % real numbers
\newcommand{\N}{\mathbb{N}}    % natural numbers

  \newcommand{\id}{\mathds{1}} % indicator with dsfont
\theoremstyle{plain}
\newtheorem{theorem}{Theorem}[section]

\newtheorem{lemma}[theorem]{Lemma}
\newtheorem{proposition}[theorem]{Proposition}

\theoremstyle{definition}
\newtheorem{definition}[theorem]{Definition}
\newtheorem{example}[theorem]{Example}

\theoremstyle{remark}
\newtheorem{remark}[theorem]{Remark}

\usepackage{tikz}
\usetikzlibrary{arrows.meta,positioning}

\begin{document}

 \begin{frontmatter} 
   \title{Testing with p*-values: Between p-values, mid p-values, and e-values} 
  \runtitle{Wang/Testing with p*-values} 

  \begin{aug} 
  \author[A]{\fnms{Ruodu} \snm{Wang}\ead[label=e1]{wang@uwaterloo.ca}} 
  % \thankstext{t1}{Some comment}
%%%%%%%%%%%%%%%%%%%%%%%%%%%%%%%%%%%%%%%%%%%%%%
%% Addresses                                %%
%%%%%%%%%%%%%%%%%%%%%%%%%%%%%%%%%%%%%%%%%%%%%%
\address[A]{Department of Statistics and Actuarial Science,
University of Waterloo,
\printead{e1}}
  \end{aug}

  \begin{abstract}
 We introduce the notion of p*-values (p*-variables), which generalizes p-values (p-variables) in several senses. The new notion has four natural interpretations: operational, probabilistic,  Bayesian, and frequentist. A main example of a p*-value is a mid p-value, which arises in the presence of discrete test statistics. A unified stochastic representation for p-values, mid p-values, and p*-values is obtained to illustrate the relationship between the three objects.  
 We study several ways of merging arbitrarily dependent or independent p*-values into one p-value or p*-value.  Admissible calibrators of p*-values to and from p-values and e-values are obtained with nice mathematical forms, revealing the role of p*-values as a bridge between p-values and e-values. The notion of p*-values becomes useful in many situations even if one is only interested in p-values, mid p-values, or e-values. In particular, deterministic tests based on p*-values can be applied to  improve some classic methods for p-values and e-values.  
  \end{abstract}



  \end{frontmatter}

%\maketitle

%\begin{abstract}
%We introduce the notion of p*-values (p*-variables), which generalizes p-values (p-variables) in several senses. The new notion has four natural interpretations: probabilistic, operational, Bayesian, and frequentist. The simplest interpretation of a p*-value is the average of several p-values. We show that there are four equivalent definitions of p*-values. The randomized p*-test is proposed, which is a randomized version of the simple p-test. Admissible calibrators of p*-values to and from p-values and e-values are obtained with nice mathematical forms, revealing the role of p*-values as a bridge between p-values and e-values. The notion of p*-values becomes useful in many situations even if one is only interested in p-values and e-values. In particular, tests based on p*-values can be applied to  improve several classic methods for p-values and e-values.  
%
%~
%
%\noindent\textbf{Keywords}: randomized test; arbitrary dependence; average of p-values; posterior predictive p-values; calibration
%\end{abstract}

 \section{Introduction}
 
 Hypothesis testing is usually conducted with the classic notion of p-values.  
We introduce the abstract  notion of \emph{p*-variables}, with \emph{p*-values} as their  realizations, defined via a simple inequality in stochastic order, in a way similar to p-variables and p-values. 
As generalized p-variables, p*-variables are  motivated by 
  mid p-values,  and they admit 
four natural interpretations: operational, probabilistic, Bayesian, and randomized testing, arising in various statistical contexts. 
% A simple interpretation of a p*-value is the average of several p-values; moreover, it is a posterior predictive p-value of \citet{M94} in the Bayesian context.

The most important and practical example of p*-values is the class of mid p-values (\citet{L52}), arising from discrete test statistics. Mid p-variables are not p-variables, but they are  p*-variables.   Discrete test statistics appear in many applications, especially when data represent frequencies or counts; see e.g., \citet{DDR18} in the context of false discovery rate control.
 Another example of discrete p-values is the conformal p-values (\citet{VGS05}); see e.g., the recent study of \citet{BCRLS21} on detecting outliers using conformal p-values. 
 Using mid p-values is one way to address discrete test statistics; another way is using randomized p-values. 
We refer to  \citet{H15} for mid and randomized p-values in multiple testing, and \citet{RHL19} for probability bounds on combining independent mid p-values based on convex order.  

In addition to mid p-values, p*-values are also naturally connected to e-values. E-values have been recently introduced to the statistical community by \citet{VW21}, and they have several advantages %(see e.g., \cite{WR20}) 
in contrast to p-values, especially via their connections to Bayes factors and test martingales (\citet{SSVV11}), betting scores (\citet{S20}), universal inference (\citet{WRB20}), anytime-valid tests (\citet{GDK20}),   conformal tests (\citet{V20}), and false discovery rate under dependence (\citet{WR22}). %see also our Section \ref{sec:72}.

In discussions where the probabilistic specification as random variables  is not emphasized, we will loosely use the term ``p/p*/e-values" for both p/p*/e-variables and their realizations,  similarly to \citet{VW20,VW21}, and this should be clear from the context.

The relationship between p*-values and mid p-values is studied in Section \ref{sec:r1-2}. We obtain a new stochastic representation for mid p-values (Theorem \ref{th:mid-p}), which unifies the classes of p-, mid p-, and p*-values. The set of p*-values is closed under several types of operations, and this closure property is not shared by that of p-values or mid p-values (Proposition \ref{prop:convex}).
Based on these results, we find that  p*-values serve as an abstract and generalized version of mid p-values which is mathematically more convenient to work with.

There are several equivalent definitions of  p*-variables: by stochastic order (Definition \ref{def:q}), by averaging p-variables (Theorem \ref{th:convex}), by conditional probability (Proposition \ref{prop:q}),  and by randomized tests (Proposition \ref{th:newdef}); each of them represents a natural statistical path to a generalization  of p-values, and these paths lead to the same mathematical object of p*-values. Moreover, a p*-value is a posterior predictive p-value of \citet{M94} in the Bayesian context. 
The   p*-test  in Section \ref{sec:pstar-test}  is a  randomized  version of the traditional p-test; the randomization is needed because p*-values are weaker than p-values.

Merging methods are useful in multiple hypothesis testing for both p-values and e-values. 
Merging several p-values or e-values are used, either directly or implicitly, in false discovery control procedures in \citet{GW04} and \citet{GS11} and generalized Bonferroni-Holm procedures (see \cite{VW20} for p-values and \cite{VW21} for e-values).  
We study merging functions for p*-values in Section \ref{sec:merg}, which turn out to have convenient structures.  In particular, we find that a (randomly) weighted geometric average of arbitrary p*-variables multiplied by $\mathrm {e}\approx 2.718$ is a p-variable (Theorem \ref{th:geom}), allowing for a simple combination  of p*-values under unknown dependence; a similar merging function for p-values is obtained by \cite{VW20}.
In the setting of merging independent p*-values, inequalities obtained by \cite{RHL19} on mid p-values can be directly applied to p*-values. 

We explore in Sections \ref{sec:5}   the connections among p-values, p*-values, and e-values by establishing results for admissible calibrators. Figure \ref{fig:1} summarizes these calibrators where the ones between p-values and e-values are obtained in \cite{VW21}. Notably, 
for an e-value $e$, $ (2e)^{-1} $ is a calibrated p*-value, which has an extra factor of $1/2$ compared to the standard calibrated p-value $e^{-1}$. 
A composition of the e-to-p* calibration  $p^*= (2e)^{-1}\wedge 1$ 
and  the p*-to-p calibration 
$p= (2 p^*)\wedge 1$  leads to  the unique admissible e-to-p calibration $p=e^{-1}\wedge1$, thus showing that p*-values serve as a bridge between e-values and p-values.

\begin{figure}
\begin{center}
\tikzstyle{bag} = [text width=5em, text centered]
\tikzstyle{end} = []
  \begin{tikzpicture}[
      mycircle/.style={
         circle,
         draw=black,
         fill=gray,
         fill opacity = 0.1,
         text opacity=1,
         inner sep=0pt,
         minimum size=40pt,
         font=\small},
               myinvisiblecircle/.style={
         circle,
         draw=black ,
         fill=gray,
         fill opacity = 0.1,
         text opacity=1,
         inner sep=0pt,
         minimum size=55pt,
         font=\small},
      myarrow/.style={-Stealth},
      node distance=1.5cm and 3.5cm
      ]
         %   \draw[      fill=gray,
      %   fill opacity = 0.08] (0,0.3) ellipse (1.4cm and 0.7cm);
      \node[myinvisiblecircle]  (c1) {}; 
      \node at (0,-0.26) {p*-value};
                  \draw[   fill=gray,
         fill opacity = 0.04] (0,0.31) ellipse (0.85cm and 0.32cm);
       \node at (0,0.31) (c4) {\scriptsize mid p-value};
      \node[mycircle,below right=of c1] (c2) {e-value}; 
      \node[mycircle,below left=of c1] (c3) {p-value}; 

    \foreach \i/\j/\txt/\p in {% start node/end node/text/position
      c1.320/c2.140/{$e= f(p^*)$, $f$ convex}/below,
       c2.130/c1.330/{$p^*=  (2e)^{-1}\wedge 1$}/above,
      c1.200/c3.50/{$p=  (2p^*)\wedge 1$}/above,     
      c3.40/c1.210/$p^*=  p$/below,    
      c2.175/c3.5/{$p =  e^{-1}\wedge 1 $}/above,
      c3.355/c2.185/$e= f(p)$/below}
       \draw [myarrow] (\i) -- node[sloped,font=\small,\p] {\txt} (\j);
       % draw this outside loop to get proper orientation of 10
  %   \draw [myarrow] (c4) -- node[sloped,font=\small,above,rotate=180] {10} (c2);
    \end{tikzpicture}
    \end{center}
    \caption{Calibration among p-values, p*-values and e-values, where $f:[0,1]\to [0,\infty]$ is   left-continuous and decreasing with $f(0)=\infty$ and $\int_0^1 f(t)\d t=1$.}
    \label{fig:1}
    \end{figure}
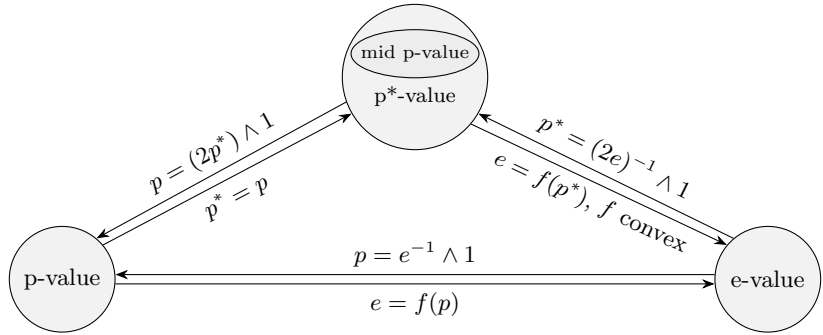

In classic statistical settings where precise (uniform on $[0,1]$) p-values are available, 
p*-values may not be directly useful as their properties are weaker than p-values. 
Nevertheless, applying a p*-test in situations where precise p-values are unavailable leads to several improvements on classic methods for p-values and e-values. 
An application on testing with e-values is discussed and numerically illustrated in Section \ref{sec:72},
and one on testing with discrete test statistics and mid p-values is presented in Section \ref{sec:r1-3}. (Another application is presented in Appendix \ref{sec:testing}.)
%We discuss and numerically illustrate such applications in detail in Sections \ref{sec:71} and \ref{sec:72}.
From these examples, we find that the tool of p*-values is useful even when one is primarily interested in p-values, mid p-values, or e-values.

%The class of p*-values enjoys many nice mathematical properties not satisfied by the class of p-values. 

The paper is written such that p-values, p*-values and e-values are  treated as  abstract measure-theoretical objects, following the setting of \cite{VW20,VW21}.  
Our null hypothesis is a generic and unspecified one, and it can be simple or composite; nevertheless, for  the discussions of our results, it would be harmless to keep a simple hypothesis  in mind as a primary example.  
All proofs are put in Appendix \ref{app:proofs} and  the randomized p*-test is discussed in Appendix \ref{sec:testing}.
 
\section{P-values, p*-values, and e-values}\label{sec:2}

%\begin{enumerate}
%\item
%\item
%\item
%\end{enumerate}
%The simplest interpretation is that p*-values are arithmetic averages of p-values. 
% 
%We introduce the notion of p*-values, which is  essentially the posterior predictive p-values of \cite{M94}.

Following  the setting of \cite{VW21}, we directly work with a fixed atomless probability space $(\Omega,\mathcal A,\p)$, where our (global) null hypothesis is set to be the singleton $\{\p\}$.
As explained in Appendix D of \cite{VW21}, no generality is lost as all mathematical results (of the kind in this paper) are valid also for general composite hypotheses. 
We assume that  $(\Omega,\mathcal A,\p)$ is  rich enough so that we can find a uniform random variable  independent of  a given random vector as we wish.  
%\section{Stochastic orders and p*-values}  
We first define stochastic orders, which will be used to formulate the main objects in the paper.
All terms like ``increasing" and ``decreasing" are in the non-strict sense. 
\begin{definition}
Let $X$ and $Y$ be two random variables. 
\begin{enumerate}
\item 
$X$ is first-order stochastically smaller than   $Y$, written as $X\le_1 Y$,  if $\E[f(X)]\le \E[f(Y)]$ for all increasing real functions $f$  such that the expectations exist. 
\item $X$ is second-order stochastically smaller than   $Y$, written as $X\le_2 Y$,  if $\E[f(X)]\le \E[f(Y)]$ for all increasing  concave real functions $f$  such that the expectations exist.
\end{enumerate}
%We also write $X=_1Y$ for $X\ge _1 Y \ge_1 X$, which is the equality in law. 
\end{definition}
%A general reference on stochastic orders is \cite{SS07}. 
%Using stochastic orders,  we define p-variables and e-variables,  and our new concept, called p*-variables.   
Recall that the defining property of a p-value, realized by a random variable $P$, is that $\p(P\le \alpha) \le \alpha$ for all $\alpha \in (0,1)$,
meaning that the type-I error of rejecting $P$ based on $P\le \alpha$ is at most $\alpha$ (see e.g., \cite{VW21}).
The above is equivalent to the statement that $P$ is first-order stochastically larger than a uniform random variable on $[0,1]$ (e.g., Section 1.A of \cite{SS07}). 
Motivated by this simple observation, we define   p-variables and e-variables,  and our new concept, called p*-variables, via stochastic orders.

%Moreover, we allow a random variable to take the value $\infty$.
\begin{definition}\label{def:q}
Let $U$ be a uniform random variable on $[0,1]$.
\begin{enumerate}
\item A random variable $P$ is a \emph{p-variable} if $U\le_1 P$. 
\item A random variable $P$ is a  \emph{p*-variable} if $U\le_2 P$.
\item A random variable $E$ is an  \emph{e-variable} if $0\le_2 E\le_2 1$.
\end{enumerate}
\end{definition}
We allow both p-variables and p*-variables to take values above one, although such values are uninteresting, and one may   safely truncate  them at $1$; moreover, we allow an e-variable $E$ to take the value $\infty$ (but with probability $0$ under the null), which corresponds to a p-variable taking the value $0$ (also with probability $0$).
 
Since $\le_1$ is stronger than $\le_2$, a p-variable is also a p*-variable, but not vice versa.
Due to the close proximity between p-variables and p*-variables, we often use $P$ for both of them; this should not create any confusion.
We refer to \emph{p-values} as realizations of p-variables,
 \emph{p*-values} as those of p*-variables, and \emph{e-values} as those of e-variables.
   By definition, both a p-variable and a p*-variable have a mean at least $1/2$.

The classic definition of an  e-variable $E$  is via $\E[E]\le 1$ and $E \ge 0$ a.s.~(\cite{VW21}).
%\begin{enumerate}
%\item $P$ is a \emph{p-variable} if  $\p(P\le \alpha) \le \alpha$ for all $\alpha \in (0,1)$;
%\item $E$ is an  \emph{e-variable} if $\E[E]\le 1$ and $E \ge 0$.
%\end{enumerate}
This is  equivalent to our Definition \ref{def:q} because
%$$ U\le_1 P~ \Longleftrightarrow ~ \p(P\le \alpha) \le \p(U\le \alpha) \mbox{~for all $\alpha \in (0,1)$};$$
$$
 0\le_2 E~\Longleftrightarrow~ 0\le E~\mbox{a.s.}; \mbox{~~~~ ~~~~}
E\le_2 1~\Longleftrightarrow~ \E[E]\le 1.$$  
We choose to express our definition via stochastic orders to make an analogy among the three concepts, %of p-values, p*-values, and e-values, 
and stochastic orders will be a main technical tool for results in this paper.

Our main focus is the notion of p*-values, which will be motivated from five perspectives in Sections \ref{sec:r1-2} and \ref{sec:2}.

\begin{remark} \label{rem:2}
There are many equivalent conditions for the stochastic order $U\le_2 P$.
One of the most convenient conditions, which will be used repeatedly in this paper,  is (see e.g., Theorem 4.A.3 of \cite{SS07})
\begin{equation}\label{eq:quantile}
U\le_2 P ~\Longleftrightarrow~ \int_0^v G_P(u) \d u  \ge \frac{v^2} 2 \mbox{~for all $v\in (0,1)$},
\end{equation} where $G_P$ is the left-quantile function of $P$, 
defined as 
$$
G_P(u)=\inf \{x\in \R: \p(P>x) \le u\}, \mbox{~~~$u\in (0,1]$}.
$$
 %The equivalence \eqref{eq:quantile} will be repeatedly used in this paper.
\end{remark}

%On the other hand, the definition of p*-variables is strongly connected to \eqref{eq:p}, which we show below.  
%\begin{remark} 
%The realization of a p*-variable is posterior predictive p-value of \cite{M94}. For this, a p*-value   could also  be called   a ppp-value or a p$^3$-value; however, I would like to stay away from the Bayesian interpretation, and as such I chose a term that does not remind me of the Bayesian context.   
%\end{remark}

\section{Mid p-values and discrete test statistics}
\label{sec:r1-2}

An important motivation for p*-values is 
the use of mid p-values and discrete test statistics. 
We first recall the usual practice to obtain p-values. 
Let $T$ be a test statistic which is a function of the observed data. 
Here, a smaller value of $T$ represents stronger evidence against the null hypothesis. 
The p-variable $P$ is usually computed from the conditional probability  \begin{equation}\label{eq:exceedance1} 
P= \p(T'\le T  |T)=F(T),\end{equation} 
where $F$ is the distribution of $T$, and $T'$ is an independent copy of $T$; here and below, a copy of  $T$ is a random variable identically distributed as $T$.
%Indeed, any p-variable as a function of the data $X$  can be obtained from \eqref{eq:exceedance1} for some $T$; this will be made rigorous in  Proposition \ref{prop:q} below.

If $T$ has a continuous distribution, then the p-variable $P$ defined by \eqref{eq:exceedance1} has a standard uniform distribution. 
If the test statistic $T$ is discrete, e.g., when testing a binomial model $\mathrm{Binomial}(n,\pi)$, 
$P$ is strictly first-order stochastically larger than a uniform random variable on $[0,1]$. 
 %The extreme case of $T$ being a constant leads to $P$ being the constant $1$.

The discreteness  of $T$   leads to a   conservative p-value; in particular, $\E[P]>1/2$. One way to address this is to randomize the  p-value to make it uniform on $[0,1]$; however, randomization is generally undesirable in testing.  
As the most natural alternative, \emph{mid p-values} (\citet{L52}) arise in the presence of discrete test statistics. 
For the test statistic $T$,  the mid p-value is given by 
 \begin{equation}\label{eq:mid}
P_T=  \frac{1} 2\p(T'\le T|T)+\frac{1}2 \p(T'<T|T)  =  \frac 12 F(T-)+ \frac 12 F(T),
\end{equation}
where $F(t-)=\lim_{s\uparrow t} F(t)$ for $t\in \R$.
%that is, the mid-point of using $\le $ and $<$  in \eqref{eq:exceedance1}.
 Clearly, $  P_T\le  F(T)$ and $\E[  P_T ] =  1/2$.  
If $T$ is continuously distributed, then $P_T=F(T)$ is uniform on $[0,1]$. 
In case $T$ is discrete, $P_T$ is not a p-variable. 
In Figure \ref{fig:r1-2} we present some examples of quantile functions of p-, mid p- and p*-variables.

Similarly to the case of p-variables in Definition \ref{def:q},
we formally define a \emph{mid p-variable} as a random variable $P\ge_1 P_T$ for some test statistic $T$. Often we have equality (i.e., $P=_1P_T$), and a strict  inequality  may appear due to, e.g., composite hypotheses, similarly to the case of p-variables or e-variables.

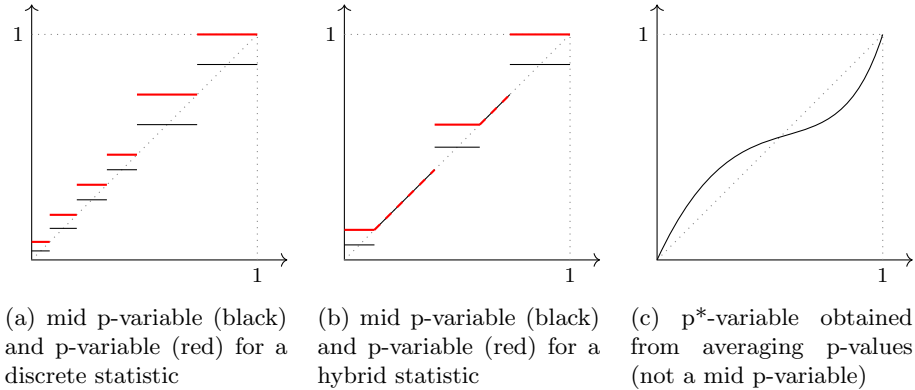
\begin{figure}[t]    
 \begin{subfigure}[b]{0.31\textwidth}
         \centering
\begin{tikzpicture}
\draw[<->] (0,3.4) -- (0,0) -- (3.4,0);
\draw[gray,dotted] (0,0) -- (3,3);
\draw[gray,dotted] (3,0) -- (3,3);
\draw[gray,dotted] (0,3) -- (3,3); 
\node[below] at (3,0) {$1$}; 
\node[left] at (0,3) {$1$}; 
\draw (0,0.12) -- (0.24,0.12);
\draw (0.24,0.42) -- (0.6,0.42);
\draw (0.6,0.8) -- (1.0,0.8);
\draw (1.0,1.2) -- (1.4,1.2);
\draw (1.4,1.8) -- (2.2,1.8);
\draw (2.2,2.6) -- (3,2.6); 
\draw[red, thick] (0,0.24) -- (0.24,0.24);
\draw[red, thick] (0.24,0.6) -- (0.6,0.6);
\draw[red, thick] (0.6,1.0) -- (1.0,1.0);
\draw[red, thick] (1.0,1.4) -- (1.4,1.4);
\draw[red, thick] (1.4,2.2) -- (2.2,2.2);
\draw[red, thick] (2.2,3) -- (3,3); 
\end{tikzpicture} 
         \caption{mid p-variable (black) and p-variable (red) for a discrete statistic}
         \label{fig:r1-2-1}
     \end{subfigure}
     ~~
      \begin{subfigure}[b]{0.31\textwidth}
         \centering
\begin{tikzpicture}
\draw[<->] (0,3.4) -- (0,0) -- (3.4,0);
\draw[gray,dotted] (0,0) -- (3,3);
\draw[gray,dotted] (3,0) -- (3,3);
\draw[gray,dotted] (0,3) -- (3,3); 
\node[below] at (3,0) {$1$}; 
\node[left] at (0,3) {$1$};  
%\node[right] at (0,3.3) {\scriptsize quantile functions}; 
\draw (0,0.2) -- (0.4,0.2);
\draw (0.4,0.4) --  (1.2,1.2);
\draw (1.2,1.5) -- (1.8,1.5);
\draw (1.8,1.8) -- (2.2,2.2);
\draw (2.2,2.6) -- (3,2.6); 
\draw[red, thick] (0,0.4) -- (0.4,0.4); 
\draw[red, thick, dashed] (0.4,0.4) --  (1.2,1.2);
\draw[red, thick] (1.2,1.8) -- (1.8,1.8); 
\draw[red, thick, dashed] (1.8,1.8) -- (2.2,2.2);
\draw[red, thick] (2.2,3) -- (3,3); 
\end{tikzpicture} 
         \caption{mid p-variable (black) and p-variable (red) for a  hybrid statistic}
        \label{fig:r1-2-2}
     \end{subfigure}
     ~~
      \begin{subfigure}[b]{0.31\textwidth}
         \centering
\begin{tikzpicture}
\draw[<->] (0,3.4) -- (0,0) -- (3.4,0);
\draw[gray,dotted] (0,0) -- (3,3);
\draw[gray,dotted] (3,0) -- (3,3);
\draw[gray,dotted] (0,3) -- (3,3); 
\node[below] at (3,0) {$1$}; 
\node[left] at (0,3) {$1$}; 
\draw [black] (0,0) .. controls (1.2,2.6) and (2.3,0.8).. (3,3);
\end{tikzpicture} 
         \caption{p*-variable obtained from averaging p-values (not a mid p-variable)}
          \label{fig:r1-2-3}
     \end{subfigure}
 
\caption{Examples of quantile functions of p-, mid p- and p*-variables}
\label{fig:r1-2}
\end{figure}

It is straightforward to verify that mid p-variables are p*-variables; see   \cite{RHL19} where  convex order is used in replace of our second-order stochastic dominance. 
The following theorem establishes a new unified stochastic representation for p-, mid p-, and p*-variables.

%\begin{theorem}\label{th:mid-p}
%Let $U$ be a standard uniform random variable  and  
%denote by $U_f:=\E[U|f(U)]$ for a measurable $f:[0,1]\to \R$.
%For any random variable $P$,
%\begin{enumerate}[(i)]
%\item $P$ is a p-variable if and only if $P\ge _1 U_f$ for a strictly increasing $f$;
%\item $P$ is a mid p-variable if and only if $P\ge _1 U_f$ for an increasing $f$;
%\item $P$ is a p*-variable if and only if $P\ge _1  U_f$ for an arbitrary $f$.  
%\end{enumerate}
%\end{theorem}

\begin{theorem}\label{th:mid-p} 
Let $U$ be a standard uniform random variable. 
For a random variable $P$,
\begin{enumerate}[(i)]
\item $P$ is a p-variable if and only if $P\ge_1  \E[U|V]$ for some $V$ which is a strictly increasing function of $U$;
\item $P$ is a mid p-variable if and only if $P\ge_1  \E[U|V]$ for some  $V$ which is an increasing function of $U$;
\item $P$ is a p*-variable if and only if  $P\ge_1  \E[U|V]$ for  some $V$ which is any random variable.
\end{enumerate}
\end{theorem}

\begin{remark}
Conditional expectations and conditional probabilities are  defined in the a.s.~sense, as usual in probability theory.
 \end{remark}

%\begin{theorem}\label{th:mid-p} 
%A random variable $P$ is  (a) a p-variable, 
%(b) a mid p-variable, or (c) a p*-variable,
% if and only if  $P\ge  \E[U|V]$ for  some standard uniform random variable $U$   and some $V$ which is  (a) a strictly increasing function of $U$, (b) 
%an increasing function of $U$, (c)  any random variable.  
%\end{theorem}

%
%\begin{proposition}\label{prop:avg}
%The variable  $  P'$ defined by  \eqref{eq:mid}
%is a p*-variable, where $T$ is any random variable, and $T'$ is an independent copy of $T$.
%\end{proposition}
%
%
%A mid p-variable is $\tilde P$ defined by \eqref{eq:mid}.
%By Proposition \ref{prop:avg}, mid p-variables form a subset of the set of all p*-variables.
%\begin{theorem}\label{th:avg}
%A random variable $P$ is a p*-variable if and only if 
%it is a distribution  mixture of some mid p-variables.
%\end{theorem}
%\begin{proof}
%The ``if" statement follows from the fact that 
%a distribution mixture of p*-variables is a p*-variable (Proposition \ref{prop:convex}), and 
%mid p-variables are p*-variables (Proposition  \ref{prop:avg}). 
%Below, we show the ``only if" statement. \com{Not true}
%\end{proof}
%

From Theorem \ref{th:mid-p}, it is clear that mid p-variables are special cases of p*-variables but the converse is not true. 
As a direct consequence, all   results later on p*-variables can be directly applied to mid p-values. 
The stochastic representations  in Theorem \ref{th:mid-p}   may not be directly useful in statistical inference; nevertheless they reveal a deep connection between mid p-values and p*-values, allowing us to analyze  possible improvements of methods designed for p*-variables when applied to mid p-values, and vice versa.

In the next result, we summarize some closure property of the sets of p-, mid p and p*-variables. 
\begin{proposition}\label{prop:convex}
\begin{enumerate}[(i)]
\item The set  of p*-variables is closed under convex combinations and under distribution mixtures.
\item The set  of p-variables is closed under distribution mixtures but not under convex combinations.
\item The set of mid p-variables is  neither closed under  convex combinations nor   under distribution mixtures.
\item The three sets above are all closed  under convergence in distribution. 

\end{enumerate}
\end{proposition}

 Proposition \ref{prop:convex} suggests  
 that the set of p*-variables has the nicest closure properties among the three. 
  Moreover, we will see in Theorem \ref{th:convex} below that the set of p*-variables is precisely the convex hull of the set of p-variables, and hence, it is also the convex hull of the set of mid p-variables which contains all p-variables. 
  
%  From  Theorem \ref{th:mid-p} and Proposition \ref{prop:convex},
%we see that  p*-values serve as an abstract and generalized version of mid p-values which is mathematically more convenient to work with.

%In contrast to  p-variables, one can use the unbiased version \eqref{eq:avg} of the probability of exceedance  to define p*-variables, thus   a midway tie-breaking in the classic formulation of p-values for discrete test statistics without any correction. 
%
\section{Four formulations of p*-values}\label{sec:2}
In this section, we will present four further equivalent definitions of p*-values, each arising from a different statistical context, and providing several interpretations of p*-values. 
 
%\subsection{Frequentist interpretation: mid p-values}

\subsection{Averages of p-values}  

Our first interpretation of p*-variables is operational: 
we will see that a p*-variable is precisely the arithmetic average of some p-values which are obtained from possibly different sources and arbitrarily dependent. 
This characterization relies on a recent technical result of \citet{MWW19} on the sum of standard uniform random variables.

\begin{theorem}\label{th:convex}
A random variable is a p*-variable if and only if 
it is the convex combination of some p-variables. 
Moreover, any p*-variable can be expressed as the arithmetic average of three p-variables. 
\end{theorem}

\begin{remark}
As implied by  Theorem 5 of  \cite{MWW19}, a p*-variable can always be written as the arithmetic average of $n$ dependent p-variables for any $n\ge 3$, but the statement  
is not true for $n=2$ (\cite[Proposition 1]{MWW19}). 
\end{remark}

As a consequence of Theorem \ref{th:convex},  the set of p*-variables is the convex hull of the set of p-variables, as briefly mentioned in Section \ref{sec:r1-2}.
Testing with the arithmetic average of dependent p-values has been studied by \cite{VW20}. 
We further 
discuss  in Appendix \ref{sec:testing} an application of p*-values which improves some tests based on arithmetic averages of p-values.
%In the next proposition we summarize some closure properties of these two  sets.% of p-variables and p*-variables. 

\subsection{Conditional probability of exceedance} 
\label{sec:31}

%To explain the motivation for p*-values, we first recall the usual practice to obtain p-values. 
%Let $T$ be a test statistic which is a function of the observed data, represented by a vector $X$. 
%Here, a smaller value of $T$ represents stronger evidence against the null hypothesis. 
%The p-variable $P$ is usually computed from the conditional probability  
Our second interpretation is probabilistic:  we will interpret both p-variables and p*-variables as conditional probabilities. 
Let $T$ be a test statistic which is a function of the observed data, represented by a vector $X$. 
Recall that a p-variable $P$ in \eqref{eq:exceedance1} has the form 
 $$ 
P=F(T)= \p(T'\le T  |T)= \p(T'\le T  |X). $$   
It turns out that p*-variables have a similar representation, where the only difference is that $T$ defining a p*-variable may not be a function of $X$; instead, $T$ can include some unobservable randomness or additional randomness in the statistical experiment (but the p*-variable itself is deterministic from the data). 
%We call this interpretation a probabilistic one as we will interpret both p-variables and p*-variables as conditional probabilities. 

%\section{Representation via conditional probability}

%In the next result, we show that both p-variables and p*-variables have a representation as conditional probabilities. 
\begin{proposition}\label{prop:q}
For a $\sigma(X)$-measurable  random variable $P$,  
\begin{enumerate}[(i)]
\item $P$ is a p-variable if and only if there exists a  $\sigma(X)$-measurable   $T$ such that $P\ge \p(T'\le T  |X)$  where $T'$ is a copy of $T $ independent of $X$;
\item $P$ is a p*-variable if and only if there exists a random variable $T$ such that $P\ge \p(T'\le T  |X)$  where $T'$ is a copy of $T $ independent of $(T,X)$.
\end{enumerate}
In (ii) above, 
$
\p(T'\le T|X) 
$ can be safely replaced  by    $\p(T'\le T|X)/2+\p(T'<T|X)/2$.
\end{proposition}

Proposition \ref{prop:q} suggests that p*-variables are very similar to p-variables when interpreted as conditional probabilities; the only difference is whether extra randomness is allowed in $T$.
 
\begin{remark}
Both Proposition \ref{prop:q} and Theorem \ref{th:mid-p} give   stochastic representations for p- and p*-variables. They are similar and with a few differences. 
First, one is stated  via stochastic order whereas  the other via inequalities between random variables.
Second, one involves a uniform random variable $U$ whereas the other does not as $T$ may be discrete.
Third, measurability conditions are different as Proposition \ref{prop:q} specifies $\sigma(X)$. 
%$T$ is $\sigma(X)$-measurable  in Proposition \ref{prop:q} ($X$ is the conditioning variable) whereas in Theorem \ref{th:mid-p} (i) and (ii), $V$ is $\sigma(U)$-measurable ($V$ is the conditioning variable).  
%Third, $P$ in Proposition \ref{prop:q} is pre-specified to be $\sigma(X)$-measurable. 
\end{remark}

\subsection{Posterior predictive p-values}
In the Bayesian context, the posterior predictive p-value of \citet{M94} is a p*-value. 
Let $X$ be the data vector in Section \ref{sec:31}. The null hypothesis $H_0$ is given by 
$\{\psi\in \Psi_0\}$ where $\Psi_0$ is a subset of the parameter space $\Psi$ on which a prior distribution is specified. 
The posterior predictive p-value   is defined as the realization of the random variable
$$
P_B:= \p(D( X',\psi) \ge D( X,\psi) | X),
$$
where $D$ is a function (taking a similar role as  test statistics), $ X'$ and $ X$  are iid conditional on $\psi$, and the probability is computed under the joint posterior distribution of $( X',\psi)$. 
Note that $P_B$ can be rewritten as 
$$
P_B=\int  \p(D( X',y) \ge D( X,y) | X,y) \d \Pi(y| X)
$$
where $\Pi$ is the posterior distribution of $\psi$ given the data $  X$. 
One can check that $P_B$ is a p*-variable by using Jensen's inequality; see Theorem 1 of \cite{M94} where   $D( X,\psi)$ is assumed to be continuously distributed conditional on $\psi$.

In this formulation, p*-variables are obtained by integrating p-variables over the posterior distribution of some unobservable parameter.
Since  p*-variables are treated as  measure-theoretic objects in this paper, we omit a detailed discussion of the Bayesian interpretation; nevertheless, it is reassuring that p*-values have a natural appearance in the Bayesian context as put forward by \citet{M94}. One of our later results is related to an observation of \cite{M94} that two times a p*-variable is a p-variable (see Proposition \ref{prop:trivial}).

 \subsection{Randomized tests with p*-values}
 \label{sec:pstar-test}

 Recall that the defining property of a p-variable $P$ is that the standard p-test
\begin{equation}\label{eq:p-test}
  \mbox{rejecting the null hypothesis} ~ \Longleftrightarrow~ P\le \alpha
 \end{equation}
 has size (i.e., probability of type-I error) at most $\alpha$ for each $\alpha\in (0,1)$.  
Since p*-values are a weaker version  of p-values, 
one cannot guarantee that the test \eqref{eq:p-test} for a p*-variable $P$ has size at most $\alpha$.
Nevertheless,  a randomized version of  the test \eqref{eq:p-test} turns out to be valid. 
Moreover, this randomized test yields a defining property for p*-variables,
just like p-variables are defined by the deterministic p-test \eqref{eq:p-test}. 

%We give a further result showing that p*-variables admit another equivalent definition: 
%    p*-variables can be defined by the randomized p*-test \eqref{eq:randomav} just like p-variables are defined by the deterministic p-test \eqref{eq:p-test}.
%  For this result, we only need the size requirement to hold for uniformly distributed $\alpha$-random thresholds.
    \begin{proposition}\label{th:newdef}
Let $V_\alpha \sim \mathrm U[0,2\alpha]$ and a random variable $P$ be independent of $ V_\alpha $, $\alpha \in (0,1/2]$.  Then $\p(P\le V_\alpha)\le \alpha$ for all $\alpha \in(0,1/2]$   if and only if $P$ is a p*-variable.
    \end{proposition}

Proposition \ref{th:newdef} implies that p*-variables are precisely test statistics which can pass the \emph{randomized p*-test} (rejection via $P\le V_\alpha$) with the specified level $\alpha$, thus a further equivalent definition of p*-variables. 
%As a consequence,  the  randomized p*-test  cannot be applied to  objects more general than the class of p*-variables. 
         The drawback of the randomized p*-test is   obvious: an extra layer of randomization is needed. This undesirable feature is the price one has to pay when a p-variable is weakened to a p*-variable.  More details and applications of the randomized p*-test are put in Appendix \ref{sec:testing}. 
         Since randomization is undesirable, we omit the detailed discussions from the main paper. 
     \begin{remark}
 The  random threshold $V_\alpha$ in Proposition \ref{th:newdef} 
 can be replaced by any random variable $V$ with mean $\alpha$ and a decreasing  density function on $(0,1)$. %For a p*-variable $P$,  the test which  rejects the null hypothesis  via $P\le V$ is called a  {randomized p*-test}.
 The uniform distribution on $[0,2\alpha]$ turns out to be the one with the smallest variance among all valid choices of the random threshold (see Proposition \ref{prop:con}).
     \end{remark}

\section{Merging p*-values and mid p-values}
\label{sec:merg}

Merging p-values and e-values is extensively studied in the literature of multiple hypothesis testing; see the recent studies \cite{LX20,VW21,VWW20} and the references therein.   
We will be interested in  merging p*-values (including mid p-values)  into both a p*-value and  a p-value.  
The following proposition gives a convenient conversion rule between p*-values and p-values.
The fact that two times a p*-variable is a p-variable is already observed by \cite{M94}.
\begin{proposition}
\label{prop:trivial}
 A p-variable is a p*-variable, and  the sum of two p*-variables is a p-variable.
 \end{proposition} 
 
 Proposition \ref{prop:trivial} implies that, in order to obtain a valid p-value from  several p*-values, a naive method is to multiply each p*-value by $2$ and then apply a valid method for merging p-values (under the corresponding assumptions).
 We will see in the next few results that we can often obtain stronger results than this.

As argued by \citet[p.50-51]{E10}, dependence assumptions are difficult to verify in multiple hypothesis testing. We will first focus on the case of arbitrarily dependent p*-values, that is, without making any assumptions on the dependence structure of the p*-variables,
and then turn to the case of independent or positively dependent p*-variables.

\subsection{Arbitrarily dependent p*-values}\label{sec:arbitrary}
%We follow the terminology in \cite{VW21}. 

% Merging   p*-values turns out to be quite pleasant to work with, specially noting that p-variables and p*-variables are very similar as seen from Proposition \ref{prop:trivial}, but merging p-values is generally quite complicated as illustrated by \cite{VW20,VWW20}.  
%

We first provide a new method which merges several p*-values into a p-value based on geometric averaging.
\citet{VW20} showed that the geometric average of p-variables 
may fail to be a p-variable, but it yields a p-variable when multiplied by $\mathrm{e}:=\exp(1)$. The constant $\mathrm e\approx 2.718$ is practically the best-possible (smallest) multiplier (see \cite[Table 2]{VW20}) that provides validity against all dependence structures.
In the next result, we show that  a similar but stronger result holds for p*-values: the geometric average of p*-variables multiplied by $\mathrm e$ is not only a p*-variable, but also a p-variable, and this  holds also for randomly weighted geometric averages. 
For using weighted p-values in multiple testing, see e.g., \citet{BH97}.
%Let $\tilde P$ be the geometric average of $K$ p*-variables. 
%Since twice a p*-variable is a p-variable, 
%the above result implies that $2\mathrm e \tilde P$ is a p-variable. 
%It turns out that this statement can be strengthened greatly in the next theorem.
% $\mathrm e \tilde P$ is a p-variable for any weighted geometric average $\tilde P$ of p*-variables.

In what follows,  the (randomly) weighted geometric average $\tilde P$  of $ P_1,\dots,P_K$ for random weights $w_1,\dots,w_K$ is given by 
$$
\tilde P= \prod_{k=1}^K P_k^{w_k} = \exp\left(\sum_{k=1}^K w_k\log P_k\right),
$$ 
where $w_1,\dots,w_K$ satisfy
\begin{equation}\label{eq:weights}
 w_1,\dots,w_K \ge 0, ~\mbox{independent of $(P_1,\dots,P_K)$, and $\sum_{k=1}^K w_k  = 1$.}
\end{equation}   
If $w_1=\dots=w_K=1/K$, then $\tilde P$ is the    unweighted geometric average of $P_1,\dots,P_K$.

\begin{theorem}
\label{th:geom}
Let $\tilde P$ be a  weighted geometric average of  p*-variables. Then $\mathrm e \tilde P$ is a p-variable.
 That is, for arbitrary p*-variables $P_1,\dots,P_K$ and   weights $w_1,\dots,w_K  $   satisfying \eqref{eq:weights},
\begin{equation}\label{eq:geom}
\p\left(\prod_{k=1}^K P_k^{w_k}\le \alpha \right) \le \mathrm e \alpha ~~~\mbox{for all $\alpha \in (0,1)$}.
\end{equation}
\end{theorem}

%The next result gives an equivalent condition for an increasing function to be a p*-merging function, and it will be useful to analyze some classes of p*-merging functions later.
%

In Theorem \ref{th:geom}, 
the random weights are allowed to be arbitrarily dependent. %We require   the sum of their  {expected values} (instead of their realized values)  not to exceed $1$. 
These random weights may come from preliminary experiments. One way to obtain such weights is to use scores such as e-values from preliminary data. %; more precisely,
%if $E_1,\dots,E_K$ are e-variables independent of $(P_1,\dots,P_K)$, then
%$w_k:= E_k/\sum_{j=1}^K E_j$ for $k=1,\dots,K$ give random weights satisfying \eqref{eq:weights}. 
Using e-values to compute   weights is quite natural as a main motivation of e-values is  an accumulation of evidence between consecutive experiments; see \cite{GDK20},  \cite{VW21} and \cite{WR22}.

%The reason why random weights are allowed is because results  like Theorem \ref{th:geom} on p*-values are built on inequalities on expectations. The same observation can be made for results on e-values, but not for methods that are specific for p-values (which usually do not involve expectations). This illustrates a subtle difference between the formulations of p-variables and p*-variables. 

Theorem \ref{th:geom} generalizes the result of \cite[Proposition 4]{VW20} which considered unweighted geometric average of p-values.  When dependence is unspecified, testing with (randomly) weighted   geometric averages of p*-values has the same critical values $\alpha/\mathrm e$ as those with unweighted p-values.  
%A similar result holds for the arithmetic average as implied by Proposition \ref{prop:p*-merg} below.

Next, we will study two methods which merges p*-values into a p*-value. Since two times p*-variable is a p-variable, probability guarantee can also be obtained from these merging functions.
A \emph{p*-merging function} in dimension $K$ is an increasing Borel function $M$ on $[0,\infty)^K$ such that  $M(P_1,\dots,P_K)$ is a p*-variable for all p*-variables $P_1,\dots,P_K$; p-merging and e-merging functions are defined analogously; see   \cite{VW21}.
A p*-merging function $M$ is \emph{admissible} if it is not strictly dominated by another p*-merging function.
  %i.e., no p*-merging function $M'$ satisfies $M'\le M$ and $M\ne M$.
\begin{proposition}\label{prop:p*-merg}
The arithmetic average $M_{K}$ is an admissible p*-merging function in any dimension $K$.
\end{proposition}

Proposition \ref{prop:p*-merg} illustrates that p*-values are very easy to combine using an arithmetic average; recall that  $M_K$ is not a valid p-merging function  since the average of p-values is not necessarily a p-value (instead, $2M_K$ is a p-merging function). On the other hand, $M_K$ is an admissible e-merging function which essentially dominates all other symmetric admissible e-merging functions (\cite[Proposition 3.1]{VW20}). 

Another benchmark merging function is the Bonferroni merging function 
$$M_B:(p_1,\dots,p_K)\mapsto \left(K\bigwedge_{k=1}^Kp_k\right)\wedge 1.$$
 The next result shows that $M_B$ is  an admissible p*-merging function.
 The  Bonferroni merging function
$M_B$ is known to be an admissible  p-merging function (\cite[Proposition 6.1]{VW21}), whereas its transformed form (via $e=1/p$) is  an e-merging function but not admissible; see \cite[Section 6]{VW21} for these claims.

\begin{proposition}\label{prop:p*-merg2}
The Bonferroni merging function $M_B$ is an admissible p*-merging function in any dimension $K$.
\end{proposition}

Combining Propositions \ref{prop:p*-merg} and \ref{prop:p*-merg2}, 
  p*-merging is admissible via both the arithmetic average (admissible for e-merging, invalid for p-merging) and the Bonferroni correction (admissible for p-merging, inadmissible for e-merging).

%Below we 

%
%\section{Merging positively dependent p*-values}
%\label{sec:r1-1} 

\subsection{Independent p*-variables}

We next turn to the problem of merging independent p*-variables. 
Merging independent mid p-values is studied by \citet{RHL19} based on arguments of convex order. Since our p*-variables are defined via the order $\ge_2$ which is closely related to convex order,  the bounds in \cite{RHL19} can be directly applied to the case of p*-variables. More precisely, for any p*-variable $P$,  using Strassen's theorem   in the form of  \cite[Theorems 4.A.5 and 4.A.6]{SS07},
there exists a random variable $Z$ such that 
$  Z\le P$ and $Z$ satisfies the convex order relation used in \cite{RHL19}.
In particular, for the arithmetic average $\bar P$ of \emph{independent} p*-variables, using \cite[Theorem 1]{RHL19} leads to the probability bound
\begin{equation}\label{eq:ind}
\p(\bar P \le 1/2-t) \le \exp(-6Kt^2) \mbox{~~~~for all $t\in [0,1/2]$.}
\end{equation}
Another  probability bound for the geometric average of independent p*-variables  is obtained by \cite[Theorem 2]{RHL19}  based on the observation that twice a p*-variable is a p-variable (cf.~Proposition \ref{prop:trivial}). Recall that Fisher's combination method uses the geometric average  of independent p-values.

It is well-known that  statistical validity  
of  Fisher's method or other methods based on concentration inequalities 
can be fragile when independence does not hold; see also our simulation results in Section \ref{sec:r1-3}.
Since independence is difficult to verify in multiple hypothesis testing (see e.g., \cite{E10}), these independence-based methods (for either p-values or p*-values) need to be applied with caution. 

There are, nevertheless, some methods which work well for independent p-values and are relatively robust to dependence assumptions.
In addition to the Bonferroni correction which is valid for all dependence structures, the most famous  such method is perhaps that of \citet{S86}.
Define the function
$$
S_K:[0,\infty)^K \to [0,\infty),~ S_K(p_1,\dots,p_k)=\bigwedge_{k=1}^K\frac {K p_{(k)}}{k},
$$
where $p_{(k)}$ is the $k$-th smallest order statistic of $p_1,\dots,p_K$.
A celebrated result of \cite{S86} is that 
 if $P_1,\dots,P_K$ are independent p-variables, then 
% $S_K(P_1,\dots,P_K)$ is a p-variable. 
 the Simes inequality holds
\begin{equation}
\label{eq:Simesineq}
\p( S_K(P_1,\dots,P_K)  \le \alpha)\le \alpha~~~\mbox{for all $\alpha \in (0,1)$}.
\end{equation}
Further, if $P_1,\dots,P_K$ are iid uniform on $[0,1]$, then $S_K(P_1,\dots,P_K)$ is again uniform on $[0,1]$.  
 The Simes inequality \eqref{eq:Simesineq} holds also under some notion of positive dependence, in particular,   positive regression dependence (PRD); see \citet{BY01} and \citet{RBWJ19}.

%Let us define the notion of positive regression dependence (PRD) in the form of \citet{RBWJ19}. 

%\begin{definition}\label{def:PRD}
%  A set $A\subseteq \R^K$ is said to be \emph{increasing} 
%if $\mathbf x\in A$ implies $\mathbf y\in A$ for  all $\mathbf y\ge \mathbf x$.  Inequalities should be
%interpreted component-wise when applied to vectors. 
%Random variables $P_1,\dots,P_K$ are PRD  if for any index $k\in \{1,\dots,K\}$   and increasing set $A  \subseteq  \R^K$, the
%function $x\mapsto \p(\mathbf P\in A\mid P_k\le x)$ is increasing on $ [0,1]$.
%\end{definition}  
%
%As shown by \citet{S98} and \citet{BY01},
%if p-variables $P_1,\dots, P_K$ are PRD, then 
%\eqref{eq:Simesineq} holds (see Lemma 2 of \cite{RBWJ19}). 

One may wonder whether  $S_K(P_1,\dots,P_K)$ yields a p*-variable or p-variable for  independent or PRD p*-variables $P_1,\dots,P_K$. It turns out that this is not the case, as illustrated by the following example,
where $S_2(P_1, P_2)$ fails to be a p*-variable or a p-variable, even in case that $P_1$ and $P_2$ are iid p*-variables.

\begin{example}
Let $P_1$ be a random variable satisfying $\p(P_1=0.2)=0.4$ and $\p(P_1=0.7)=0.6$. 
It is straightforward to verify that $P_1$ is a p*-variable (indeed, it is a mid p-variable by Theorem \ref{th:mid-p}).
Let $P_2$ be an independent copy of $P_1$ and $P:=S_2(P_1,P_2)$. 
We can check that 
$\p(P=0.2)=0.16$, $\p(P=0.4)=0.48$ and $\p(P=0.7)=0.36$. 
It follows that $\E[P]=0.476<1/2$, and hence $P$ is not a p*-variable.
\end{example}
 
 Since twice a p*-variable is a p-variable (Proposition \ref{prop:trivial}),  
it is safe (and conservative) to use $2S_K(P_1,\dots,P_K)$  which is a p-variable under independence or PRD (note that PRD is preserved under linear transformations). 

Other methods on p-values that are relatively robust to dependence include the harmonic mean p-value of \citet{W19} and the Cauchy combination method of \citet{LX20}.
As shown by \citet{CLTW22}, the three methods of Simes, harmonic mean, and Cauchy combinations are closely related and  similar in several senses. 

%All methods in Section \ref{sec:arbitrary}  for p*-variables are valid for all dependence structures. 
Obviously,  more robustness  to dependence leads to a more conservative method. 
Indeed, all p-merging methods designed for arbitrary dependence are quite conservative in some situations; see the comparative study in \cite{CLTW22}.
Thus, there is a trade-off between power and robustness to dependence.
%this observation remains true for p*-values.
%In the case of p*-variables, a similar trade-off appears naturally.
For p*-merging methods, the bound \eqref{eq:ind} is the most stringent on the independence assumption.
Using $2S_K$ is valid for independent or PRD p*-variables.
Finally, all methods in Section \ref{sec:arbitrary} work for any dependence structure among the p*-variables.

\begin{remark}
Any function $M$ which merges iid standard uniform random variables $U_1,\dots,U_K$  into a standard uniform one, such as the functions in the methods of Simes, Fisher's and the Cauchy combination,
satisfies
\begin{equation}
\label{eq:sharp}
\p(M(P_1,\dots,P_K)\le \alpha)\le \alpha \mbox{~~for all $\alpha \in(0,1)$}
\end{equation}
for any independent p-variables $P_1,\dots,P_K$.
However, they generally cannot satisfy \eqref{eq:sharp}
for all independent p*-variables (or mid p-variables) $P_1,\dots,P_K$, since  $M(P_1,\dots,P_K)\ge_1 M(U_1,\dots,U_K)$ does not hold for some choices of $P_1,\dots,P_K$. Therefore, some form of penalty always needs to be paid when relaxing p-values to p*-values  or mid p-values for these methods.
\end{remark}
 
  \begin{remark}
 The function $S_K$ and the inequality \eqref{eq:Simesineq} play a central role in multiple hypothesis testing and false discovery rate (FDR) control; in particular, the procedure of \citet{BH95} at level $\alpha$
reports at least one discovery for p-values $p_1,\dots,p_K$ if and only if 
 $ S_K(p_1,\dots,p_K) \le  \alpha $, and \eqref{eq:Simesineq} guarantees the FDR of this procedure is no longer than $\alpha$ in the global null setting with independent p-values.
 \end{remark}

\section{Calibration between  p-values,  p*-values, and e-values}
\label{sec:5}

In this section, we discuss calibration between p-, p*-, and e-values. 
Calibration between p-values and e-values is one of the main topics of \cite{VW21}.  

\subsection{Calibration between  p-values and p*-values}
Calibration between p-values and p*-values is relatively simple. 
A \emph{p-to-p* calibrator} is an increasing function $f:[0,1]\to [0,\infty)$ that transforms p-variables to p*-variables,
and a \emph{p*-to-p calibrator} is an increasing function $g:[0,1]\to [0,\infty)$ which transforms in the reverse direction. 
Clearly, the values of p-values larger than $1$ are irrelevant, and  hence we restrict the domain of all calibrators in this section  to be $[0,1]$; in other words, input p-variables and p*-variables larger than $1$ will be treated as $1$.
A calibrator is said to be \emph{admissible} if it is not strictly dominated by another calibrator of the same kind (for calibration to p-values and p*-values,  $f$ dominates $g$ means $f \le g$, and for calibration to  e-values in Section \ref{sec:6} it is the opposite inequality).

%\begin{proposition}\label{prop:Meng}
%A p-variable is a p*-variable, and  the sum of two p*-variables is a p-variable.
%\end{proposition}
%\begin{proof}
%
%\end{proof} 

\begin{theorem}\label{prop:calibrator}
\begin{enumerate}[(i)] 
\item 
The  p*-to-p calibrator $u\mapsto (2u)\wedge 1$ dominates all other p*-to-p calibrators.
\item 
An increasing function $f$ on $[0,1]$ is an admissible p-to-p* calibrator if and only if $f$ is left-continuous, $f(0)=0$, $\int_0^v f(u) \d u  \ge  v^2/2 $ for all $v\in (0,1)$, and $\int_0^1 f(u) \d u  = 1/2 $.
\end{enumerate}
\end{theorem}

Theorem \ref{prop:calibrator} (i) states that a multiplier of $2$ is the best calibrator that works for all p*-values. 
This observation justifies  the deterministic threshold $\alpha/2$ in the test  \eqref{eq:halfalpha} for p*-values, as mentioned in Section \ref{sec:pstar-test}.   
Although Theorem \ref{prop:calibrator} (ii) implies that there are many admissible p-to-p* calibrators, it seems that there is no obvious reason to use anything other than the identity in Proposition \ref{prop:trivial} when calibrating from p-values to p*-values. %On the other hand, when calibrating from p*-values to p-values, the extra factor of $2$ needs to be applied. 
Finally, we note that the conditions in  Theorem \ref{prop:calibrator} (ii)  imply that the range of $f$ is contained in $[0,1]$, an obvious requirement for an admissible p-to-p* calibrator.

\subsection{Calibration between p*-values and e-values}
\label{sec:6}

Next, we discuss calibration between e-values and p*-values, which has a richer structure.
A \emph{p*-to-e calibrator} is a  decreasing function $f:[0,1]\to [0,\infty]$ that transforms p*-variables to e-variables,
and an \emph{e-to-p* calibrator} $g:[0,\infty]\to [0,1]$  is a  decreasing function which  transforms in the reverse direction.  We include $e=\infty$ in the calibrators, which corresponds to $p=0$.

%, and both $e=\infty$ and $p=0$ occur with probability $0$ under $\p$.
%Again,  we restrict the domain of p*-to-e calibrators to be $[0,1]$.

First,   since a p-variable is a p*-variable,
any p*-to-e calibrator is also a p-to-e calibrator. 
Hence, the set of p*-to-e calibrators is contained in the set of p-to-e calibrators. 
By Proposition 2.1 of \cite{VW21}, 
any admissible p-to-e calibrator $f:[0,1]\to [0,\infty]$ is a decreasing  function such that $f(0)=\infty$, $f$ is left-continuous, and $\int_0^1 f(t) \d t=1$. We will see below that some of these admissible p-to-e calibrators are also p*-to-e calibrators.

Regarding the other direction of e-to-p* calibrators, we first recall that there is a unique admissible e-to-p calibrator,  given by  $e\mapsto  e^{-1}\wedge 1$, as shown by \cite{VW21}.
Since the set of p*-values is larger than that of p-values, the above e-to-p calibrator is also an e-to-p* calibrator. The interesting questions are whether there is any e-to-p* calibrator stronger than $e\mapsto e^{-1}\wedge 1$, and whether an admissible e-to-p* calibrator is also unique.  
The constant map $e\mapsto 1/2$ is an e-to-p* calibrator since $1/2$ is a constant p*-variable. If there exists an  e-to-p* calibrator $f$ which dominates all other e-to-p* calibrators, then it is necessary that $f(e)\le 1/2$ for all $e\ge 0$; however this would imply $f=1/2$ since any p*-variable has mean at least $1/2$. Since $e\mapsto 1/2$ does not dominate $e\mapsto e^{-1}\wedge 1$, we conclude that there is no  e-to-p* calibrator which dominates all others, in contrast to the case of e-to-p calibrators.

Nevertheless, some refined form of domination can be helpful.
We say that an  e-to-p* calibrator $f$ \emph{essentially dominates} 
another e-to-p* calibrator $f'$ if  $f(e)\le f'(e)$ whenever $f'(e)<1/2$.
That is, we only require dominance when the calibrated p*-value is useful (relatively small);
this consideration is similar to the essential domination of e-merging functions in \cite{VW21}.  
It turns out that the e-to-p calibrator $e\mapsto e^{-1}\wedge 1$ can be improved by a factor of $1/2$, which essentially dominates all other e-to-p* calibrators.

The following theorem summarizes the validity and admissibility results on both directions of calibration.  

\begin{theorem}\label{prop:pstar-e}
\begin{enumerate}[(i)]
\item 
A convex (admissible) p-to-e calibrator is an (admissible) p*-to-e calibrator.
% \item A convex   admissible p-to-e calibrator is an admissible p*-to-e calibrator.
\item  
An admissible p-to-e calibrator is a p*-to-e calibrator if and only if it is convex.
\item The e-to-p* calibrator $e\mapsto  (2e)^{-1}\wedge 1$ essentially dominates all other  e-to-p* calibrators.
\end{enumerate}
\end{theorem}

All practical examples of p-to-e calibrators  are convex and admissible; see \cite[Section 2 and Appendix B]{VW21} for a few classes (which are all convex). 
 By Theorem \ref{prop:pstar-e}, all of these calibrators are admissible p*-to-e calibrators. A popular class of p-to-e calibrators is  given by, for   $\kappa\in (0,1)$, 
\begin{equation}\label{eq:pecali}
p\mapsto \kappa p^{\kappa-1},~~~~p\in [0,1].
\end{equation}
Another simple choice, proposed by \citet{S20}, is 
\begin{equation}\label{eq:shafer}
p\mapsto   p^{-1/2}-1,~~~~p\in [0,1].
\end{equation}
Clearly, the  p-to-e calibrators  in \eqref{eq:pecali} and \eqref{eq:shafer} are convex and thus they are p*-to-e calibrators.

%Nevertheless, we have an optimality result with a slight and natural modification. 

%\begin{theorem}\label{prop:e-pstar}
%The e-to-p* calibrator $e\mapsto  (2e)^{-1}\wedge 1$ essentially dominates all other  e-to-p* calibrators.
%\end{theorem}

The result in Theorem \ref{prop:pstar-e} (iii)   shows that the unique admissible e-to-p calibrator 
$e\mapsto e^{-1}\wedge 1$ can actually be achieved by a two-step calibration:
first use $p^*=  (2e)^{-1}\wedge 1$ to get a p*-value, and 
then use $p=(2p^*)\wedge 1$ to get a p-value.

On the other hand, all p-to-e calibrators $f$ in \cite{VW21} are convex, and they can be seen as a composition of the calibrations $p^*=p$ and $e=f(p^*)$. 
Therefore, p*-values serve as an intermediate step in both directions of calibration between p-values and e-values, although one of the directions is less   interesting since the p-to-p* calibrator is an identity.  
Figure \ref{fig:1} in the Introduction illustrates our recommended calibrators among  p-values, p*-values and e-values based on Theorems \ref{prop:calibrator} and \ref{prop:pstar-e},
and they are all admissible.

\begin{example}\label{ex:pecali}
%Let us look at e-values obtained from applying the p-to-e calibrators \eqref{eq:pecali} to p-variables.
Suppose that $U$ is uniformly distributed on $[0,1]$. Using the calibrator \eqref{eq:pecali}, for $\kappa \in (0,1)$,  $E:=\kappa U^{\kappa-1}$ is an e-variable.
By Theorem \ref{prop:pstar-e} (iii)  , we know that $P:=(2E)^{-1} $ is a p*-variable. 
Below we check this directly. 
 The left-quantile function $G_P$  of $P $ satisfies
$$ 
G_P(u)= \frac{u^{1-\kappa}}{2\kappa },~~~~u\in (0,1).
$$
Using $\kappa (2-\kappa) \le 1$ for all $\kappa \in (0,1)$, we have
$$
\int_0^v G_P(u) \d u=\frac{v^{2-\kappa}}{2\kappa(2-\kappa)} \ge \frac{v^{2-\kappa}}2\ge \frac{v^2}2,~~~~v\in (0,1).
$$
Hence,  $P$ is a p*-variable by   verifying \eqref{eq:quantile}.
Moreover, for $\kappa \in (0,1/2]$, $P$ is even a p-variable, since $G_P(u)\ge u$ for $u\in (0,1)$.
\end{example}

In the next result, we show that   a  p*-variable  obtained from the calibrator in Theorem \ref{prop:pstar-e} (iii)    
is a p-variable under   a further condition (DE):
    \begin{enumerate}
    \item[(DE)]  $E\le_1 E'$ for some e-variable $E'$ which has a decreasing density on $(0,\infty)$.
    % or it is first-order stochastically smaller than an e-variable with a decreasing density on $(0,\infty)$. 
    \end{enumerate}
    In particular, condition (DE) is satisfied if $E$ itself has a decreasing density on $(0,\infty)$. 
    Examples of e-variables satisfying (DE) are those obtained from applying a non-constant convex p-to-e calibrator $f$ with $f(1)=0$ to any p-variable, e.g., the p-to-e calibrator \eqref{eq:shafer} but not \eqref{eq:pecali}; this is because convexity of the calibrator yields a decreasing density when applied to a uniform p-variable. 
\begin{proposition}\label{prop:1over2e}
For any e-variable $E$, $P:=(2E)^{-1}\wedge 1 $ is a p*-variable,
and if $E$ satisfies (DE), then $P$ is a p-variable.
\end{proposition}

  \begin{remark}
 In a spirit similar to Proposition \ref{prop:1over2e},   smoothing techniques leading to an extra factor of $2$ in the Markov inequality have been studied by \citet{H19}. 
 %In particular, Corollary 1 of \cite{H19} is a special case of Proposition \ref{prop:1over2e}.
 \end{remark}
 
%
%\section{Applications}
%\label{sec:7}

\section{Testing with e-values and  martingales}
\label{sec:72}

    In this section we discuss  applications of  p*-values to tests with e-values and martingales. E-values and test martingales are usually used for purposes more than rejecting a null hypothesis while controlling type-I error; in particular, they offer anytime validity and different interpretations of statistical evidence (e.g., \cite{GDK20}). We compare the power of several methods here for a better understanding of their performance, while keeping in mind that single-run detection power (which is maximized by p-values if they are available) is not the only purpose of e-values.
    
Suppose that $E$ is an e-variable, usually obtained from likelihood ratios or stopped test supermartingales (e.g., \cite{SSVV11}, \cite{S20}). 
A traditional e-test is 
  \begin{align} \label{eq:e-test}
 \mbox{rejecting the null hypothesis}   ~\Longleftrightarrow~    E  \ge  \frac{1}{\alpha}.
     \end{align}  
Using the fact that $ (2E)^{-1}$ is a p*-variable in Theorem \ref{prop:pstar-e} (iii)  , we can design the randomized test 
  \begin{align} \label{eq:e-test1p}
 \mbox{rejecting the null hypothesis}   ~\Longleftrightarrow~   2E  \ge  \frac{1}{V},
     \end{align} 
     where $V\sim \mathrm{U}[0,2\alpha]$ is independent of $E$ (Proposition \ref{th:newdef}).
  %  Just like the comparison between \eqref{eq:robustav}  and \eqref{eq:randomav1},
   The test  \eqref{eq:e-test1p}  has $3/4$ chance of being more power than the traditional choice of testing $E$ against $1/\alpha$ in \eqref{eq:e-test}. Randomization is undesirable, but \eqref{eq:e-test1p} inspires us to look for alternative deterministic methods.
     
Suppose that one has two independent e-variables $E_1$ and $E_2$ for a null hypothesis.
     As shown by \cite{VW21}, it is  optimal in a weak sense to use the combined e-variable $E_1E_2$ for testing the null.    Assume further that one of $E_1$ and $E_2$ satisfies condition (DE).  

Using \eqref{eq:e-test1p} with the random threshold $\alpha E_2$ and Proposition \ref{prop:storder} in Appendix \ref{sec:testing}, we get 
    $ \p ((2E_1)^{-1} \le \alpha E_2 ) \le \alpha $ (note that the positions of $E_1$ and $E_2$ are symmetric here).
    Hence,  the test 
      \begin{align} \label{eq:e-test2}
 \mbox{rejecting the null hypothesis}   ~\Longleftrightarrow~   2E_1E_2  \ge  \frac{1}{\alpha},
     \end{align} 
     has  size at most $\alpha$. 
    The threshold of the test \eqref{eq:e-test2} is half  the one obtained by directly applying \eqref{eq:e-test} to the e-variable $E_1E_2$. Thus, the test statistic is boosted by a factor of $2$ via   condition (DE) on either $E_1$ or $E_2$.
No assumption is needed for  the other e-variable.
In particular, by setting  $E_2=1$, we get a p-variable $(2E_1)^{-1}$  if $E_1$ satisfies (DE), as we see in Proposition \ref{prop:1over2e}.
     
E-values calibrated from p-values are useful in the context of testing randomness online (see \cite{V20}) and designing test martingales  (see \cite{DRBW19}).
% and \cite{VW20c}).
More specifically,
for a possibly infinite sequence of independent p-variables $(P_t)_{t\in \N}$ and a sequence of p-to-e calibrators $(f_t)_{t\in \N}$, 
the stochastic process 
$$
X_t= \prod_{k=1}^t f_k (P_k),~~~~t=0,1,\dots
$$
is a supermartingale (with respect to the filtration of $(P_t)_{t\in \N}$) with initial value $X_0=0$ (it is a martingale if $P_t$, $t\in \N$ are standard uniform and $f_t$, $t\in \N$ are admissible).  
As a supermartingale, $(X_t)_{t\in \N}$ satisfies the anytime validity, i.e.,
$X_\tau$ is an e-variable for any stopping time $\tau$; moreover, Ville's inequality gives
\begin{align}\label{eq:e-test-max}
\p\left(\sup_{t\in \N} X_t \ge \frac 1 \alpha\right)\le \alpha \mbox{~for any  $\alpha>0$.}
\end{align}
The process $(X_t)_{t\in \N}$ is called an \emph{e-process} by \cite{WR22}.
 Anytime validity is crucial in the design of online testing where evidence arrives sequentially in time, and scientific discovery is reported at a stopping time     considered with sufficient evidence. 

Notably, the most popular choice of p-to-e calibrators   are those in \eqref{eq:pecali} and \eqref{eq:shafer} (see e.g., \cite{V20}), which are convex. Theorem \ref{prop:pstar-e} implies that
if the inputs  are not p-values but p*-values, we can still obtain e-processes using convex calibrators such as \eqref{eq:pecali} and \eqref{eq:shafer}, without calibrating these p*-values to p-values. This observation becomes useful when each observed $P_t$ is only a p*-variable, e.g., a mid p-value  or an average of several p-values from parallel experiments.

Moreover, for a fixed $t\in \N$, if there is a convex $f_s$ for some $s\in\{1,\dots,t\}$ with $f_s(1)=0$, and $P_s$ is a p-variable (the others can be p*-variables with any   p*-to-e calibrators),
then  (DE) is satisfied by $f_s(P_s)$, and we have $\p(X_t\ge 1/\alpha)\le \alpha /2$ by using the test \eqref{eq:e-test2}; see our numerical experiments below.
% we can use averages of p-values under arbitrary dependence to calibraotr
     
%     Suppose next that one has an e-variable $E$ (representing the level of prior evidence) and an independent p-variable $P$ (representing the result of a new experiment) for the same null hypothesis. 
%      Applying \eqref{eq:ppstar-test} to the p*-variable $ (2E)^{-1}$, a combined test is given by 
%       \begin{align}\label{eq:ep-test} 
% \mbox{rejecting the null hypothesis}   ~\Longleftrightarrow~   4E (1-P)   \ge  \frac{1}{\alpha}.
%     \end{align} 
%     The combined test \eqref{eq:ep-test} has half the threshold as that obtained from directly applying \eqref{eq:e-test} to the e-variable $2E(1-P)$. 
%                                 Alternatively,  applying \eqref{eq:ppstar-test2} to the p*-variable $ (2E)^{-1}$, we obtain
%                        \begin{align}\label{eq:ep-test3} 
% \mbox{rejecting the null hypothesis}   ~\Longleftrightarrow~   \frac{ 2-1/E } {P}   \ge  \frac{1}{\alpha} .
%     \end{align}    
%     Finally, note that $\p(P\le \alpha E)\le \alpha$, we can design the test
%              \begin{align}\label{eq:ep-test2} 
% \mbox{rejecting the null hypothesis}   ~\Longleftrightarrow~   \frac{E}{P}  \ge  \frac{1}{\alpha}.
%     \end{align}   
%     It is disappointing to notice that \eqref{eq:ep-test2} always has a better power than both \eqref{eq:ep-test} and \eqref{eq:ep-test3}  because $1/P \ge 4(1-P)$ and $E\ge 2 -1/E$.  
%    This shows a loss of evidence in the e-to-p* calibration $e\mapsto (2e)^{-1}$.

    \subsection*{Simulation experiments}

      In the simulation results below,
      we generate  test martingales  following \cite{VW21}.
      Similarly to Section \ref{sec:71}, the null hypothesis   $H_0$ is $ \mathrm{N}(0,1)$ and the alternative is $ \mathrm{N}(\delta,1)$ for some $\delta>0$. We generate iid $X_1,\dots,X_n$ from $\mathrm{N}(\delta,1)$. 
Define the e-variables from the likelihood ratios of the alternative to the null density,
\begin{equation}\label{eq:base-e}
  E_t
  :=
  \frac{\exp(-(X_t-\delta)^2/2)}{\exp(-X_t^2/2)}
  =
  \exp(\delta X_t - \delta^2/2),~~~t=1,\dots,n.
\end{equation}
The e-process $S=(S_t)_{t=1,\dots,n}$ is defined as 
$
S_t=\prod_{s=1}^t E_s. 
$
Such an e-process $S$ is \emph{growth optimal} in the sense of \citet{S20}, as it maximizes the expected log growth among all test martingales built on the data $(X_1,\dots,X_n)$; indeed, $S$ is Kelly's strategy under the betting interpretation. Here, we constructed the e-process $S$  assuming that we know $\delta$; otherwise we can use universal test martingales (e.g., \cite{HRMS20}) by taking a mixture of $S$ over $\delta$ under some probability measure.

Note that each $E_t$ is log-normally distributed and it does not satisfy (DE). Hence, \eqref{eq:e-test2} cannot be applied to $S_n$.
Nevertheless, we can replace $E_1$ by  another e-variable $E_1'$
which satisfies (DE).  We choose $E_1'$ by applying the p-to-e calibrator \eqref{eq:shafer} to the p-variable $P_1=  1-  \Phi(X_1) $, namely,
$
E_1'= (P_1)^{-1/2}-1.
$   
 
 Replacing $E_1$ by $E_1'$, we obtain the new e-process $S'=(S'_t)_{t=1,\dots,n}$ by 
 $
S'_t=E_1'\prod_{s=2}^t E_s. 
$
The e-process $S'$ is not growth optimal, but 
as $E_1'$ satisfies (DE), we can test via the rejection condition $2 S'_n \ge 1/\alpha$, thus boosting the terminal value by a factor of $2$. Let $V\sim \mathrm{U}[0,2\alpha]$ be independent of the test statistics.
We compare five different tests, all with size at most $\alpha$:
\begin{enumerate}[(a)]
\item applying \eqref{eq:e-test}   to $S_n$: reject $H_0$ if $ S_n \ge 1/\alpha$ (benchmark case);
\item applying \eqref{eq:e-test1p} to $S_n$: reject  $H_0$  if $2  S_n \ge 1/V $;
\item applying \eqref{eq:e-test2}   to $S'_n$: reject $H_0$  if $2 S'_n \ge 1/\alpha$;
\item applying a combination of  \eqref{eq:e-test1p}  and \eqref{eq:e-test2}   to $S'_n$: reject  $H_0$ if $2 S'_n \ge 1/V$;
\item applying \eqref{eq:e-test-max} to the   maximum of $S$: reject  $H_0$ if $ \max_{1\le t\le n} S_t \ge 1/V$.
\end{enumerate}
Since test (a) is strictly dominated by test (e), we do not need to use (a) in practice; nevertheless we treat it as a benchmark for comparison on tests based on e-values as it is built on the fundamental connection between e-values and p-values: the e-to-p calibrator $e\mapsto e^{-1}\wedge 1$.

The significance level  $\alpha$ is set to be $ 0.01$. 
The power of the five tests 
is  computed from the average of 10,000 replications for varying signal strength $\delta$ and for $n\in \{2, 10,100\}$. Results are reported in Figure \ref{fig:2}. 
For most values of $\delta$ and $n$, either the deterministic test (c) for $S'$ or the maximum test (e) has the best performance.
The deterministic test (c) performs very well in the cases $n=2$ and $n=10$, especially for weak signals; this may be explained by the factor of $2$ being substantial when the signal is weak.
If $n$ is large and the signal is not too weak,   the effect of using the maximum of $S$ in (e) is dominating; this is not surprising.  
Although the randomized test (b) usually improves the performance from the benchmark case (a),
the advantages seem to be quite limited, especially in view of the extra randomization, often undesirable.   
\begin{figure}[t]
\begin{center}
 \makebox[\textwidth]{\includegraphics[width=4.2cm, trim=0 40 20 20, clip]{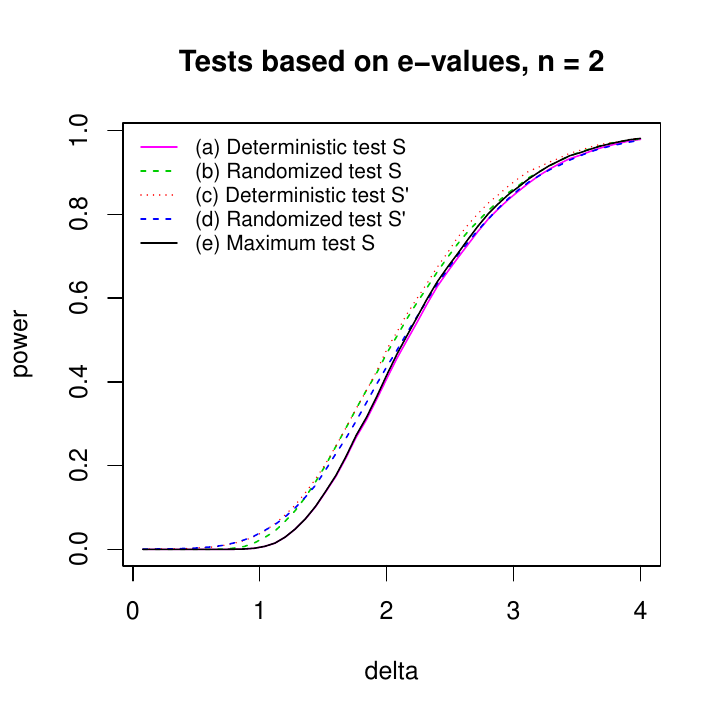}
 \includegraphics[width=4.2cm, trim=0 40 20 20, clip]{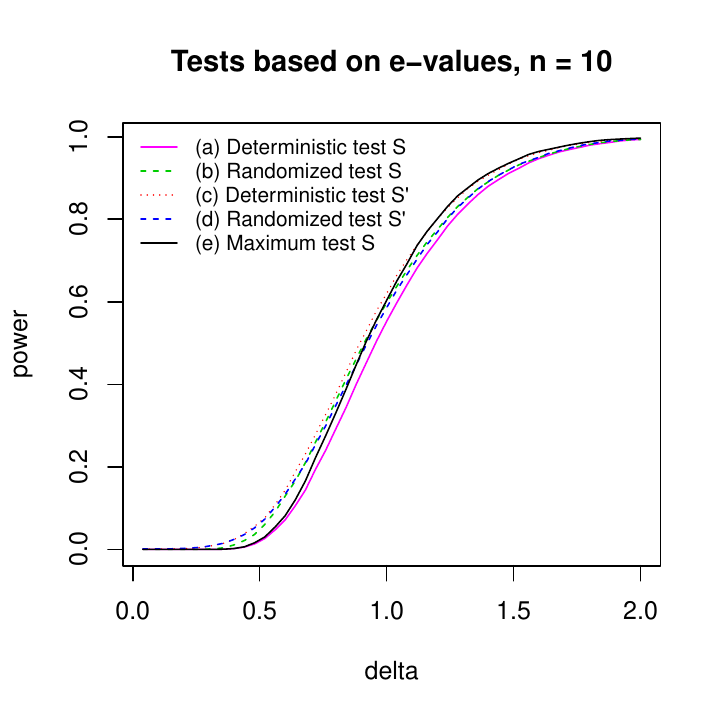}
  \includegraphics[width=4.2cm, trim=0 40 20 20, clip]{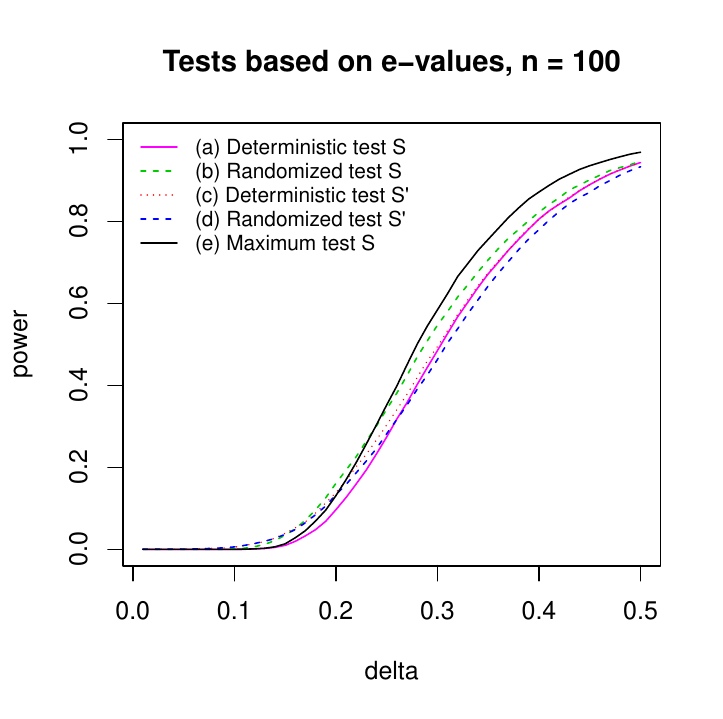}} \\
   \makebox[\textwidth]{\includegraphics[width=4.2cm, trim=0 10 20 50, clip]{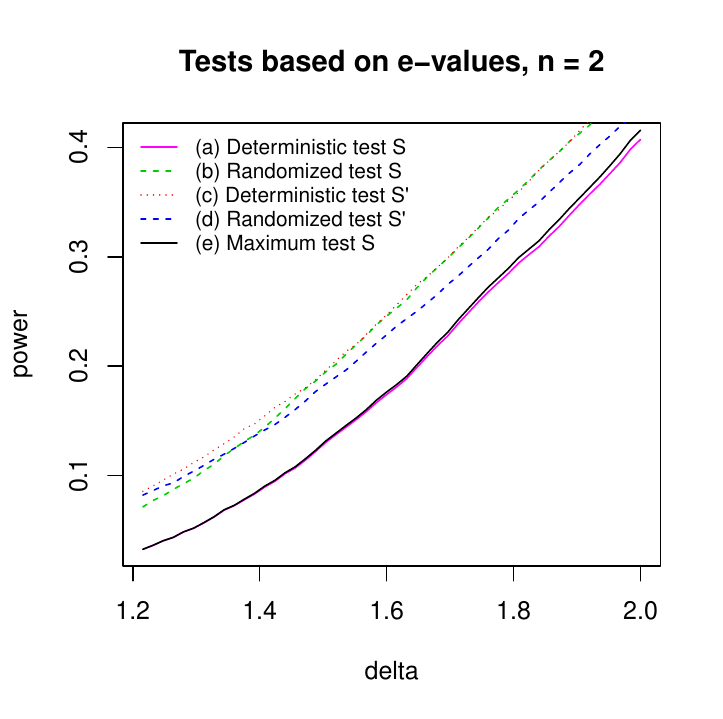}
 \includegraphics[width=4.2cm, trim=0 10 20 50, clip]{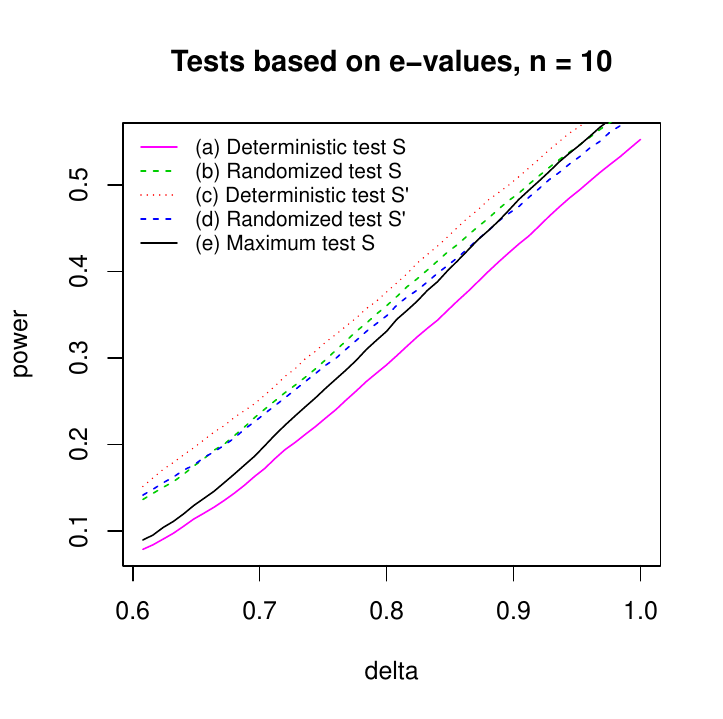}
  \includegraphics[width=4.2cm, trim=0 10 20 50, clip]{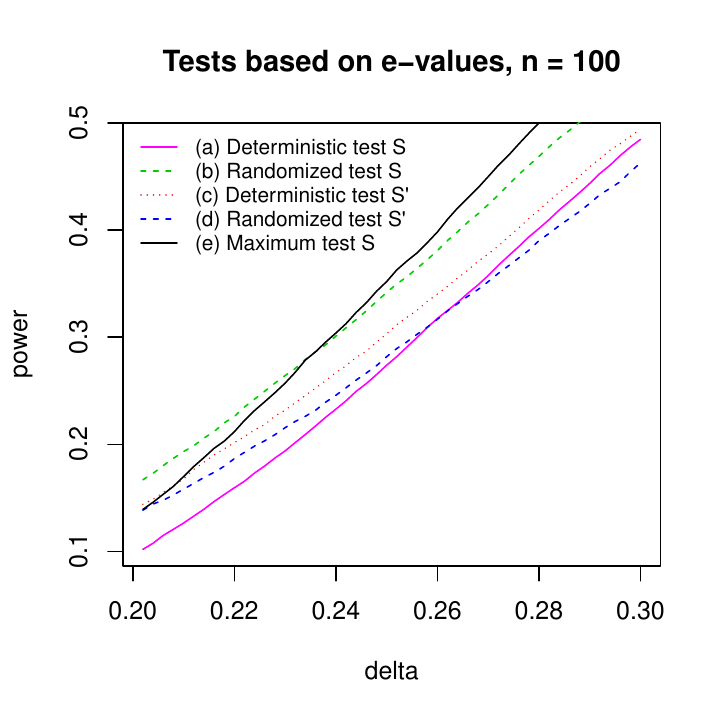}}
\caption{Tests based on e-values; the second row is zoomed in from the first row}
\label{fig:2}
\end{center}
\end{figure}

\section{Testing with combined mid p-values}
\label{sec:r1-3}
We compare  by simulation the performance of a few tests via merging mid p-values. 
The (global) null hypothesis $H_0$ is that the test statistic $T$ follows
a binomial distribution  $\mathrm{Binomial}(n,\pi)$, and $K$ tests are conducted. 
We set $n=40$ and $\pi=0.3$, so that the obtained p-values are 
considerably discrete. 
We denote by  $P_1,\dots,P_K$ the obtained p-values via  \eqref{eq:exceedance1},
and $P^*_1,\dots,P^*_K$ the obtained mid p-values via \eqref{eq:mid}.
Let $\bar P^*$ and $\tilde P^*$ be the arithmetic average and the geometric average of $
 P^*_1,\dots,P^*_K$, respectively. 
 
The true data probability generating 
the test statistics is a binomial distribution $\mathrm{Binomial}(n,(1-\theta)\pi)$,
where $\theta\in [0,1]$. The case $\theta=0$ means that the null hypothesis $H_0$ is true, and 
a larger $\theta$ indicates a stronger signal. 
 
 We allow the test statistics to be correlated, and this is achieved by simulating from a Gaussian copula with common pairwise correlation parameter $\rho$ (more precisely, we first simulate from a Gaussian copula, and then obtain the observable discrete test statistics by a quantile transform).  We consider the following tests (there are   other tests possible for this setting, and we only compare these four to illustrate a few relevant points):
\begin{enumerate}[(a)]
\item the probability bound  \eqref{eq:ind} on the arithmetic mean  in \cite{RHL19}: reject $H_0$ if $ \bar P^* \le  1/2-(-\log(\alpha)/(6 K))^{1/2} $;
\item the arithmetic mean times $2$ using Proposition \ref{prop:p*-merg}: reject $H_0$ if $ \bar P^*\le \alpha/2$;
\item the geometric average of p*-values using Theorem \ref{th:geom}: reject  $H_0$  if $ \tilde P^* \le  \alpha/\mathrm{e}$;
\item the Bonferroni correction: reject $H_0$ if $ \min(P_1,\dots,P_K) \le  \alpha/K$.
%\item the Simes method: reject $H_0$ if $ \min_k( K P_{(k)}/k) \le  \alpha $ where $P_{(k)}$ is the $k$-th smallest p-value;
\end{enumerate}
Note that tests (a), (b) and (c) use mid p-values  based on methods for p*-values, and (d) uses p-values. 
All of (b), (c) and (d) are valid tests under arbitrary dependence (AD) whereas the validity of (a) requires independence.
Therefore, we expect (a) to perform very well in case independence holds. All other methods are  valid but conservative, as there is a big price to pay to gain robustness against all dependence structures.

The significance level  $\alpha$ is set to be $ 0.05$ for good visibility.
The power of the four tests is reported in Figure \ref{fig:r1-1} which 
is  computed from the average of 10,000 replications for varying signal strength $\theta\in [0,1]$  and  for $\rho=0$ (independence), $\rho=0.2$ (mild dependence) and $\rho=0.8$ (strong dependence). The situation of $\rho=0.8$ is the most relevant to us as   averaging methods are designed mostly for situations where  the presence of strong or complicated dependence is suspected. 
\begin{figure}[t]
\begin{center}
 \makebox[\textwidth]{\includegraphics[width=4.2cm, trim=0 40 20 20, clip]{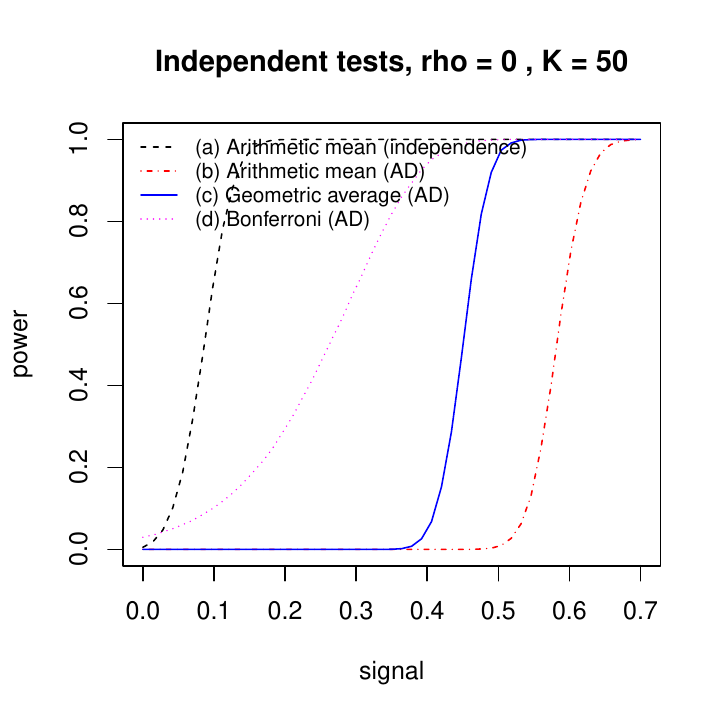}
 \includegraphics[width=4.2cm, trim=0 40 20 20, clip]{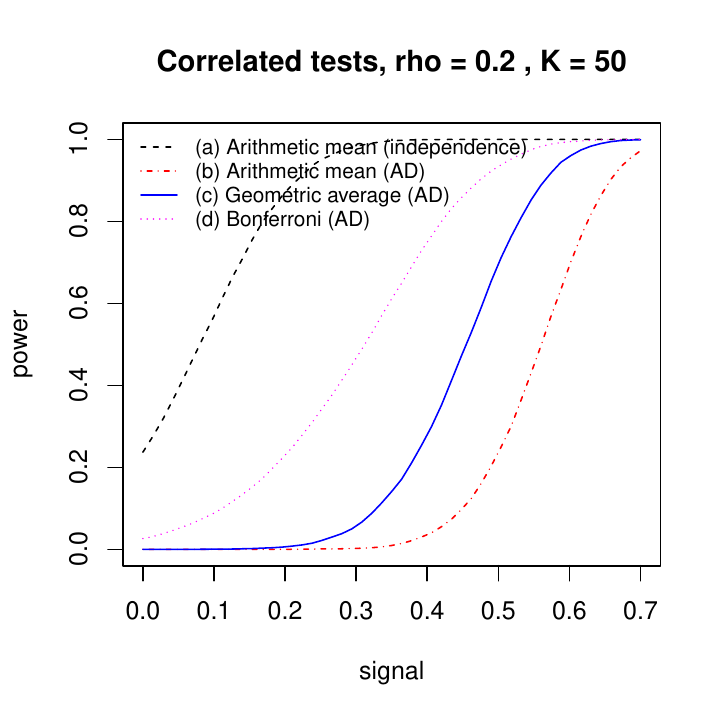}
  \includegraphics[width=4.2cm, trim=0 40 20 20, clip]{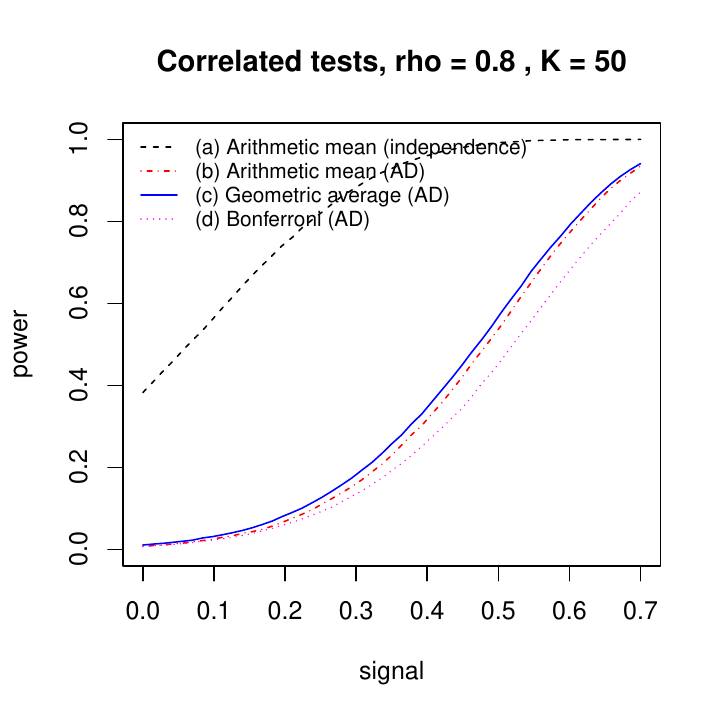}} \\
 \makebox[\textwidth]{\includegraphics[width=4.2cm, trim=0 10 20 0, clip]{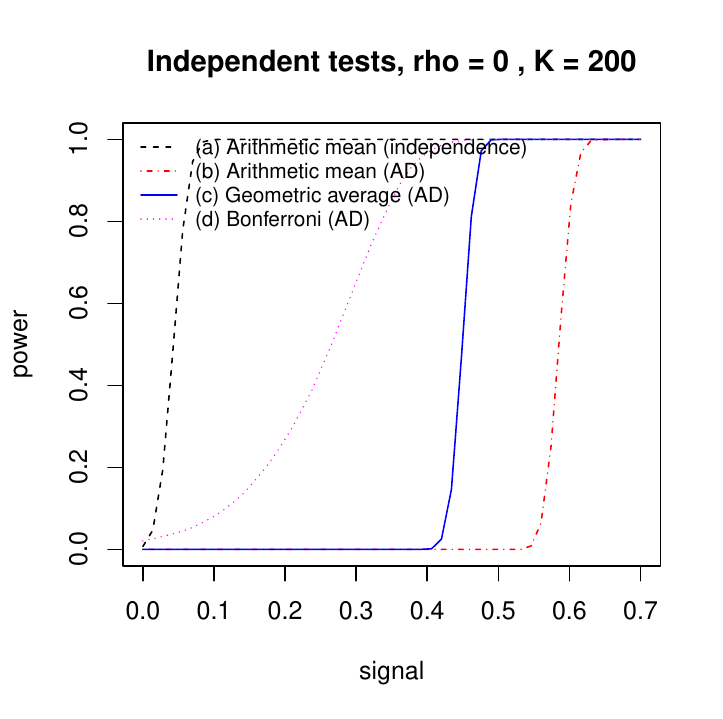}
 \includegraphics[width=4.2cm, trim=0 10 20 0, clip]{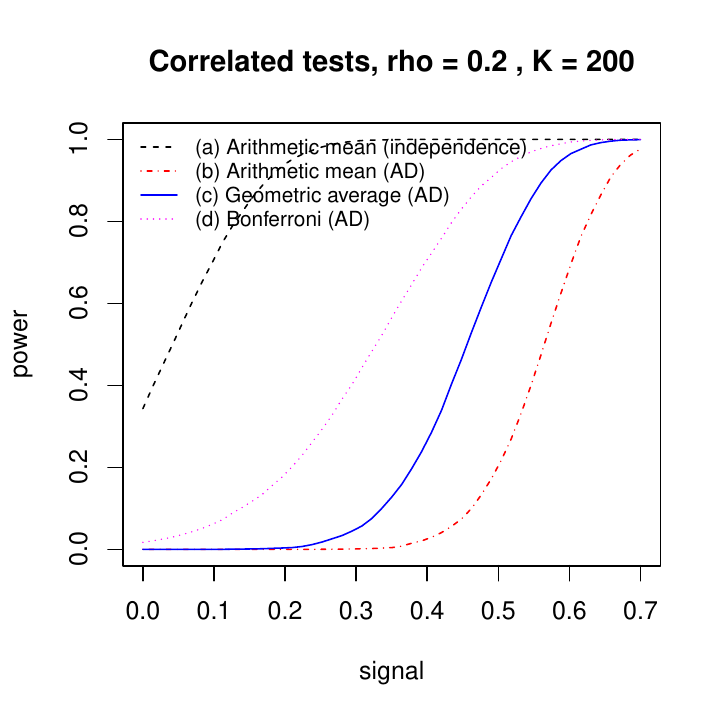}
  \includegraphics[width=4.2cm, trim=0 10 20 0, clip]{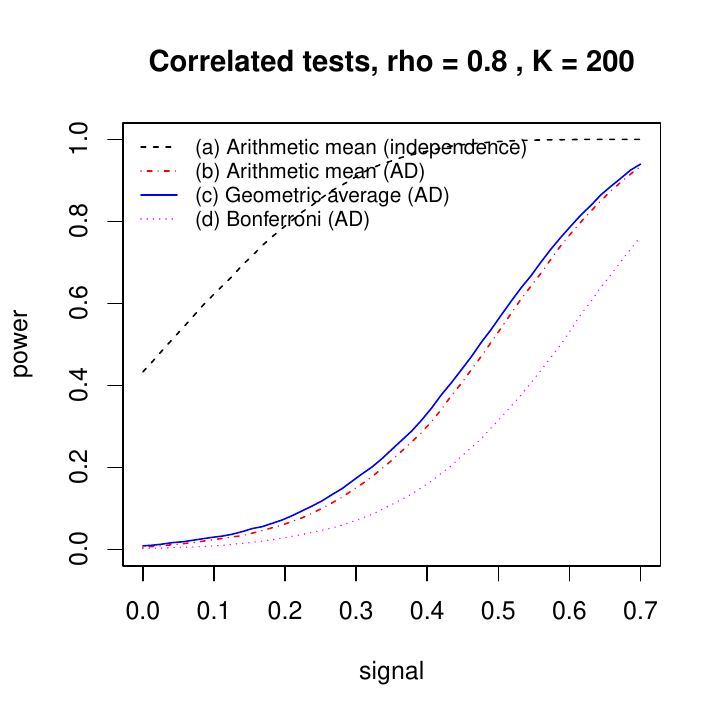}} \\  
\caption{Tests based on combining p-values or mid p-values}
\label{fig:r1-1}
\end{center}
\end{figure}

As we can see from Figure \ref{fig:r1-1}, 
the test (a) relying on independence has the strongest power as expected.
Its size becomes largely inflated as soon as mild dependence is present, and hence 
it can only  be applied in situations where independence among obtained mid p-values can be justified.
Indeed,   the size of test (a) can be $\approx 0.4$ in case $\rho=0.2$, which is clearly not useful.
Among the three methods that are valid for arbitrary dependence,
the geometric average (c) has stronger power for large $\rho$ 
and the Bonferroni correction (d) has stronger power for small $\rho$. 
The arithmetic average   (b) performs  poorly unless p-values are very strongly correlated.  
These observations on merging mid p-values are consistent with those in \cite{VW20} on merging p-values.   
In conclusion, the geometric average (c) can be useful when  the dependence among mid p-values is speculated to be strong or complicated, although unknown to the decision maker.

\section{Conclusion}

In this paper we introduced   p*-values (p*-variables) as an abstract measure-theoretic object.
The notion of p*-values is a manifold generalization of p-values, and it enjoys many attractive theoretical properties in contrast to p-values. 
In particular, mid p-values, which arise in the presence of discrete test statistics, form an important subset of p*-values. 
Merging methods for p*-values are studied. In particular, a weighted geometric average of arbitrarily dependent p*-values multiplied by $\mathrm e\approx 2.718$ yields a valid p-value, which can be useful when multiple mid p-values are possibly strongly correlated. 

Results on calibration between p*-values and e-values reveal that p*-values serve as an intermediate step in both the standard  e-to-p  and p-to-e calibrations.
%We discuss testing with p*-values and introduced the randomized p*-test. 
Although  a  direct test with p*-values may involve randomization, we find that p*-values are useful in the design of deterministic tests with averages of p-values, mid p-values, and e-values. 
From the results in this paper, the concept of p*-values serves   a useful technical tool which enhances the extensive and growing applications of p-values, mid p-values, and e-values. 

%\subsubsection*{Acknowledgements}

\begin{acks}[Acknowledgments]
The author thanks %an Editor, an Associate Editor, three anonymous referees, 
Ilmun Kim,  Aaditya Ramdas,  and Vladimir Vovk   for constructive comments on an earlier version of the paper,
and Tiantian Mao, Marcel Nutz, and Qinyu Wu for kind help on  some technical statements. 
\end{acks}

%%%%%%%%%%%%%%%%%%%%%%%%%%%%%%%%%%%%%%%%%%%%%%
%% Funding information, if any,             %%
%% should be provided in the                %%
%% funding section.                         %%
%%%%%%%%%%%%%%%%%%%%%%%%%%%%%%%%%%%%%%%%%%%%%%
\begin{funding}
The author acknowledges financial support from   and the Natural Sciences and Engineering Research Council of Canada (RGPIN-2018-03823, RGPAS-2018-522590).
\end{funding}

\appendix
\section{Proofs of all results}
\label{app:proofs}
In this appendix, we collect proofs of all theorems and propositions in the main paper.

\subsection{Proofs of results in Section \ref{sec:r1-2}}

%\begin{proof}[Proof of Proposition \ref{prop:avg}]
%Let $U$ be a uniform random variable on $[0,1]$ that $G_T(U)=T$ a.s., where $G_T$ is the left-quantile function of $T$. Let $U'$ be an iid copy of $U$. 
%We have 
%$$
%\p(U'\le U|T) =  \frac{1} 2\p(T'\le T|T)+\frac{1}2 \p(T'<T|T) = \tilde P.
%$$
%Hence, by Proposition \ref{prop:q} (replacing $X$ therein by our $T$), $\tilde P$ is a p*-variable.
%\end{proof} 

\begin{proof}[Proof of Theorem \ref{th:mid-p}]
%In the   statements below, we will use the fact that 
%if $P\ge_1 X$, then we can find $Y$ identically distributed as $X$ such that 
%$
%P\ge Y.
%$
% Theorem 1.A.1 of \cite{SS07} guarantees the existence of $P'$ and $X'$ on another probability space which is rich such that $P'\ge X'$,
% and $P'$ and $X'$ are identically distributed to $P$ and $X$, respectively.
% We can directly construct $Y$ satisfying $P\ge Y$ on the original probability 
% space because it is assumed to be rich enough for such a construction.  
\begin{enumerate}[(i)]
\item For $V$ being a strictly increasing function of $U$, we have $\E[U|V]=U$, 
and the equivalence statement follows directly from the definition of p-variables. 
%and the ``if" statement is straightforward. The ``only if" statement follows from   the aforementioned fact that $P\ge_1 U'$ for a standard uniform $U'$ implies the existence of  a standard uniform  $U$ such that $P\ge U$.
\item We first show the ``if" statement. Write $V=f(U)$ for an increasing function $f$. Denote by $F$ and $G$  the distribution function and the  left-quantile function of $V$, respectively, and let $Z$ be a standard uniform random variable independent of $V$.
Moreover, let 
$$
U_V = F(V-)+Z(F(V)-F(V-)),
$$
which is uniformly distributed on $[0,1]$ and satisfies $G(U_V)=V$ a.s.~(e.g., \cite[Proposition 1.3]{R13}).
Since $G(U_V)=V=f(U)$ a.s., and both $G$ and $f$ are increasing,  
we know that the functions $G$ and $f$ differ on a set of Lebesgue measure $0$. 
%Let $U_T$ be a uniform random variable such that 
%$G(U_T)=T$ a.s.~(e.g., \cite[Lemma A.32]{FS16}).
 Therefore, $\E[U|V] = \E[U  | G(U)]$, which is identically distributed as 
 $\E[U_V|G(U_V)]$. Moreover, 
\begin{equation}
 \E[U_V|G(U_V)] = \E[U_V|V] = \frac 12 F(V-)+ \frac 12 F(V). \label{eq:r1-mid-p}
\end{equation}
Therefore, $P\ge_1 \E[U|V]$ implies    $P\ge_1  (F(V-)+    F(V) )/2 $, and thus $P$ is a mid p-variable.

The ``only if" statement follows from  $P\ge_1 (F(T-)+F(T))/2$ and 
  \eqref{eq:r1-mid-p} by choosing $U=U_V$ and $V=T$.

%Next, we show the ``only if" statement.
%Note again that, since our probability space is rich,  $P\ge_1 (F(T-)+F(T))/2$ implies that there exists a random variable $V$ such that  $P\ge (F(V-)+F(V))/2$ a.s.
%The desired inequality $P\ge \E[U|V] $ follows from \eqref{eq:r1-mid-p} by choosing $U=U_V$.
\item  We first show the ``if" statement.
Note that $P\ge \E[U|V] \ge_2 U$ 
where the second inequality is guaranteed by Jensen's inequality. 
Since $\ge $ is stronger than $\ge_2$ and   $\ge_2$  is  transitive, 
we get $P\ge_2 U$ and hence $P$ is a p*-variable.

Next, we show the ``only if" statement. By using  \cite[Theorems 4.A.5]{SS07},
the definition of a p*-variable $P$ implies that there exist a standard uniform random variable $U$  and a random variable $V$  such that 
$ P \ge_1 V \ge \E[U|V]$. \qedhere
\end{enumerate}
\end{proof}

\begin{proof}[Proof of Proposition \ref{prop:convex}]
By Theorem \ref{th:convex},  the set of p*-variables is the convex hull of the set of p-variables, and thus convex. 
This also implies that none of the set of p-variables or that of mid p-variables is convex.
%It is well known that the set of p-variables is not convex.
 
To show that the set of p*-variables is closed under distribution mixtures, it suffices to note that the stochastic orders  $\le_1$ and $\le_{2}$ (indeed, any order induced by inequalities via integrals) is closed under distribution mixture.

To see that the set of mid p-variables is not closed under distribution mixtures,
we note from \eqref{eq:mid} that 
any mid p-variable with mean $1/2$ and a point-mass at $1/2$
must not have any density in a neighbourhood of $1/2$.
Hence, the mixture of uniform distribution on $[0,1]$ and a point-mass at $1/2$ is not the distribution of a mid p-variable.

Closure under convergence for $\le_1$ is justified  by Theorem 1.A.3 of \cite{SS07},
and closure under convergence for $\le_2$ is justified by Theorem 1.5.9 of \cite{MS02}.  
 Closure under convergence for the set of mid p-values follows by noting that the set of distributions of $P_T$ in \eqref{eq:mid} is closed under convergence in distribution, which can be checked by definition.
%Proposition 3.3 of \cite{MW15}. 
\end{proof}

\subsection{Proofs of results in Section \ref{sec:2}}

\begin{proof}[Proof of Theorem \ref{th:convex}]
We first show that a convex combination of p-variables is a p*-variable. 
Let $U$ 
be a uniform random variable on $[0,1]$, $P_1,\dots,P_K$ be $K$ p-variables, $(\lambda_1,\dots,\lambda_K)$ be an element of the standard $K$-simplex, and $f$ be an increasing concave function.
 By monotonicity and concavity of $f$, we have
$$
\E\left[f\left( \sum_{k=1}^K \lambda_k  P_k \right) \right]
\ge \E \left[  \sum_{k=1}^K \lambda_k  f(P_k)\right] \ge \E \left[ \sum_{k=1}^K \lambda_k f(U)\right] = \E[f(U)].
$$
Therefore,  $\sum_{k=1}^K \lambda _k P_k\ge_2 U$ and thus  $ \sum_{k=1}^K \lambda _k P_k $ is a p*-variable. 
%Since p-variables are p*-variables, we conclude that a convex combination of p-variables is a p*-variable.

Next, we show the second statement that any p*-variable can be written as the average of three p-variables, which also justifies the ``only if" direction of the first statement.  

Let $P$ be a p*-variable satisfying  $\E[P]=1/2$. 
Note that $P\ge_2 U$ and $\E[P]=\E[U]$ together implies $P\ge_{\rm cv} U$ (see e.g., \cite[Theorem 4.A.35]{SS07}), where $\le_{\rm cv}$ is the concave order, meaning that $\E[f(P)]\ge \E[f(U)]$ for all concave $f$. Theorem 5 of \cite{MWW19} says that any $P\ge_{\rm cv} U$, there exist  three standard uniform random variables $P_1,P_2,P_3$ such that 
$3P=P_1+P_2+P_3$ (this statement is highly non-trivial). This implies that $P$ can be written as the arithmetic average of three p-variables $P_1,P_2,P_3$.

Finally, assume that the p*-variable $P$ satisfies $\E[P]>1/2$. 
In this case, using Strassen's Theorem in the form of  \cite[Theorems 4.A.5 and 4.A.6]{SS07},
there exists a random variable $Z$ such that 
$U\le_{\rm cv} Z\le P$. 
As we explained above,  there exist p-variables $P_1,P_2,P_3$ such that 
$3Z=P_1+P_2+P_3$.
For  $i=1,2,3$, let $\tilde P_i:=P_i+(P-Z) \ge_1 P_i$. Note that $\tilde P_1,\tilde P_2,\tilde P_3$ are p-variables and $3P= \tilde P_1+ \tilde P_2+ \tilde P_3$.
Hence, $P$ can be written as the arithmetic average of three p-variables. 
\end{proof}

\begin{proof}[Proof of Proposition \ref{prop:q}]
For the ``only-if" statement in (i), 
since $P$ is a p-variable, we know that its distribution $F$ satisfies $F(t)\le t$ for $t\in (0,1)$. Therefore, by setting $T=P$,
$$P \ge F(P) = F(T) = \p(T' \le T | T) = \p(T' \le T|X),$$
where the last equality holds since $T'$ is independent of $X$. 
To check the ``if" direction of (i),  we have $\p(T'\le T |X ) =\p(T'\le T|T) = F(T)$ where $F$ is the distribution of $T$.  Note that $F(T)$ is stochastically larger than or equal to  a uniform random variable on $[0,1]$, and hence $\p(F(T)\le t) \le t$.   

Next, we show (ii).
First, suppose that  $P \ge \p(T'\le T  |X) $.
Let $U$ be a uniform random variable on $[0,1]$. 
By Jensen's inequality, 
we have $\E[F(T)|X] \ge_2  F(T)$.
Hence, 
$P\ge_2 \p(T'\le T  |X) = \E[F(T)|X] \ge_2  F(T)  \ge_2 U$, and thus $P$ is a p*-variable.

For the converse direction,  suppose that $P$ is a p*-variable.  
By Strassen's Theorem in the form of  \cite[Theorem 3.A.4]{SS07}, 
there exists a uniform random variable $U'$ on $[0,1]$ and a random variable $P'$ identically distributed as $ P$ 
such that 
$P'\ge \E[U'|P']$. 
Let $G(\cdot|p)$ be the left-quantile function of a regular conditional distribution of $U'$ given $P'=p \in [0,1]$. 
Further, let $V$ be a uniform random variable on $[0,1]$ independent of  $(P,X)$, and $U:=G(V|P)$.
It is clear that $(U,P)$ has the same law as $(U',P')$.
Therefore, $P\ge \E[U|P]$. Moreover, $\E[U|X] = \E[G(V|P)|X] = \E[G(V|P)|P]$ since $V$ is independent of $X$.
Hence, $P\ge \E[U|P]=\E[U|X]$.  
Let $V'$ be another uniform random variable  on $[0,1]$ independent of $(U,X)$. 
We have
$$
\p(V'\le U|X) = \E\left [\id_{\{V'\le U\}} |X \right] = \E[U|X].
$$ 
Hence, the representation $P\ge \p(T'\le T|X)$ holds with $T'=V'$ and $T=U$. 

Finally, we show the last statement on replacing
$
\p(T'\le T|X) 
$  by    $\p(T'\le T|X)/2+\p(T'<T|X)/2$.
The ``only-if" direction follows from  the  argument for (ii) by noting that $U$ constructed above has a continuous distribution. 
The ``if" direction follows from 
\begin{align*}
\frac12 (\p(T'\le T|X)+\p(T'<T|X))  & =\frac12 (\E[F(T)|X] + \E[F(T-)|X]) \\& \ge_2  \frac12  (F(T)+ F(T-)) \ge_2 U,
\end{align*}
 where the second-last inequality is Jensen's, and the last inequality is  implied by Theorem \ref{th:mid-p}. 
\end{proof}

\begin{proof}[Proof of Proposition \ref{th:newdef}]
The ``if" statement is implied by Proposition \ref{th:1}.
To show the ``only if" statement, denote by $F_P$ the distribution function of $P$ and $F_U$ be the distribution function of a uniform random variable $U$ on $[0,1]$.
We have
$$
\alpha \ge \p(P\le V_\alpha) = \int_0^{2\alpha} \frac{F_P(u) }{2\alpha }\d u.
$$
Therefore, for $v\in (0,1]$, we have 
$$
 \int_0^{v}  {F_P(u) } \d u \le \frac{v^2}{2 } = \int_0^v u \d u =\int_0^v F_U(u)\d u.
$$
By Theorem 4.A.2 of \cite{SS07}, the above inequality implies $U\le_2 P$. 
Hence, $P$ is a p*-variable.
\end{proof} 

\subsection{Proofs of results in Section \ref{sec:merg}}

\begin{proof}[Proof of Proposition \ref{prop:trivial}]
Let $U$ be a uniform random variable on $[0,1]$.  
The first statement is trivial by definition. For the second statement, let $P_1,P_2$ be two p*-variables.   
For any $\epsilon \in (0,1)$,  
\begin{align*}
\p( P_1+P_2\le  \epsilon) &=  \E[\id_{\{P_1+P_2\le \epsilon\}}] \\&\le \E\left[\frac1{\epsilon}{(2\epsilon-P_1-P_2)_+} \right] 
 \\ &\le \E\left[\frac1{\epsilon} {(\epsilon-P_1)_+} \right] +\E\left[\frac1{\epsilon} {(\epsilon-P_2)_+}\right] 
 \\&\le 2\E\left[\frac1{\epsilon} {(\epsilon-U)_+} \right]    =  \epsilon,
\end{align*}
where the last inequality is because $U\le_{\rm 2} P_1,P_2$ and $u\mapsto (\epsilon-u)_+$ is  convex and decreasing.
Therefore, $P_1+P_2$ is a p-variable.
\end{proof}

The following lemma is needed in the proof of Theorem \ref{th:geom}.
\begin{lemma}
\label{lem:quantile}
Let $M$ be an increasing Borel function  on $[0,\infty)^K$.
Then
$M(P_1,\dots,P_K)$ is a p-variable for all p*-variables $P_1,\dots,P_K$ if and only if 
\begin{equation}\label{eq:quantile-eq}
\inf\{q_1( M(\alpha P_1,\dots,\alpha P_K)) :P_1,\dots,P_K\ge_2 U\}\ge \alpha ~\mbox{for all $\alpha \in (0,1)$,}
\end{equation}
where $U$ is uniform on $[0,1]$ and $q_1(X)$ is the essential supremum of a random variable $X$.
\end{lemma}

\begin{proof}%[Proof of Lemma \ref{lem:quantile}]
Let $g(\alpha)$ be the critical value for testing with $M(P_1,\dots,P_K)$,
that is, the largest value such that 
$$
\p\left( M(P_1,\dots,P_K) < g(\alpha) \right) \le \alpha \mbox{~~~for all p*-variables $P_1,\dots,P_K$}.
$$ 
Converting between the distribution function and the quantile function,
this means
$$
g(\alpha) =\inf\{q_\alpha( M(  P_1,\dots,  P_K)) :P_1,\dots,P_K\ge_2 U\},
$$
where $q_\alpha (X)$ be the left $\alpha$-quantile of a random variable $X$.
For an increasing function $M$,  its infimum $\alpha$-quantile can be converted to the essential supremum of random variables with conditional distributions on their lower $\alpha$-tail; see the proof of \cite[Theorem 3]{LW21} which deals with the case of additive functions (see also the proof of \cite[Proposition 1]{CLTW22} where this technique is used). 
Note that for $P\ge 2 U$, its lower $\alpha$-tail conditional distribution dominates $\alpha U$.
This argument leads to
$$
g(\alpha) = \inf\{q_1( M(  P_1,\dots,  P_K)) :P_1,\dots,P_K\ge_2 \alpha U\}.
$$
Noting that $P\ge_2 \alpha U$ is equivalent to $ P /\alpha\ge_2 U$, we obtain \eqref{eq:quantile-eq}.
\end{proof}

\begin{proof}[Proof of Theorem \ref{th:geom}] 
First, suppose that $w_1,\dots,w_K$ are non-negative constants adding up to $1$.
Let $g(\alpha)$ be the critical value for testing with $\tilde P$,
that is, the smallest value such that 
$$
\p\left(\prod_{k=1}^K P_k^{w_k}\le g(\alpha) \right) \le \alpha \mbox{~~~for all p*-variables $P_1,\dots,P_K$}.
$$  
We will show that $g(\alpha)\ge \alpha/\epsilon$.
Using Lemma \ref{lem:quantile}, we get
$$
g(\alpha)  
= \alpha \inf \left\{ q_1\left(\prod_{k=1}^K P_k^{w_k}\right):P_1,\dots,P_K\ge_2   U\right\}.
$$ 
For any p*-variables $P_1,\dots,P_K$, since $\log(\cdot)$ is an increasing concave function, we have 
\begin{align*}
 q_1\left(\sum_{k=1}^K w_k \log(P_k)\right) 
&\ge\E \left[\sum_{k=1}^K w_k \log(P_k)\right]
\ge \sum_{k=1}^K w_k \E \left[ \log(U)\right]=-1.
\end{align*} 
Therefore, $\log  ({g(\alpha)}/{\alpha })  \ge -1$, leading to the desired bound $g(\alpha)\ge \alpha /\mathrm {e}$,  and thus $$
\p \left (\prod_{k=1}^K P_k^{w_k}\le \alpha/ \mathrm e  \right) \le
\p \left (\prod_{k=1}^K P_k^{w_k}\le g(\alpha)  \right) \le
 \alpha.$$
For random $w_1,\dots,w_k$, taking an expectation leads to \eqref{eq:geom}.
\end{proof}

\begin{proof}[Proof of Proposition \ref{prop:p*-merg}]
%Denote by $M_{K}$ the arithmetic average function in dimension $K$, and 
The validity of $M_K$ as  a p*-merging function is implied by Theorem \ref{th:convex}. 
To show its admissibility, suppose that there exists a p*-merging function $M$ that strictly dominates $M_K$.
Let $P_1,\dots,P_K$ be iid uniform random variables on $[0,1]$.
The strict domination implies 
$M\le M_K$ and 
$\p(M(P_1,\dots,P_K)<M_K(P_1,\dots,P_K))>0$. 
We have
$$\E[ M(P_1,\dots,P_K) ] < \E [M_K(P_1,\dots,P_K)] = \frac 12 .
$$
This means that $M(P_1,\dots,P_K)$ is not a p*-variable, a contradiction. 
\end{proof}

\begin{proof}[Proof of Proposition \ref{prop:p*-merg2}]
Let $P_1,\dots,P_K$ be p*-variables, and $P$ 
be a random variable such that the distribution of $P$ is the equally weighted mixture of those of $P_1,\dots,P_K$. Note that 
  $P$  a p*-variable by Proposition \ref{prop:convex}. 
Let $P_{(1)} = \bigwedge_{k=1}^K P_k$. 
Using the Bonferroni inequality, we have, for any $\epsilon \in (0,1)$, 
\begin{equation}\label{eq:Bon}
\p(P_{(1)} \le \epsilon) \le \sum_{k=1}^K \p(P_k\le \epsilon) = K \p(P\le \epsilon).
\end{equation}
Let $G_1$ be the left-quantile function of $P_{(1)}$ and $G_2$ be that of $P$. 
By \eqref{eq:Bon}, we have
$G_1 (K t) \ge G_2(  t)$ for all $t\in (0,1/K)$. 
Hence, for each $y\in (0,1/K)$, using the equivalent condition  \eqref{eq:quantile},  we have 
$$
 \int_0^y K G_1 (t) \d t \ge  \int_0^y K G_2 (  t/K) \d t =K^2  \int_0^{y/K} G_2(t) \d t \ge K^2 \frac{y^2}{2 K^2 } = \frac{y^2} 2. 
$$
This implies, via the equivalent condition \eqref{eq:quantile} again, that $K P_{(1)}$ is a p*-variable.

Next we show the admissibility of $M_B$ for $K\ge 2$, since the case $K=1$ is trivial. Suppose that there is a p*-merging function $M$ which strictly dominates $M_B$.
Since $M$ is increasing,  there exists $p\in (0,1/K]$ such that $q:=M(p,\dots,p)< M_B(p,\dots,p) = Kp$. 
First, assume $2 Kp \le 1$.  
Define identically distributed random variables $P_1,\dots,P_K$ by
$$
P_k= p \id_{A_k}  + \id_{A_k^c},~~~k=1,\dots,K,
$$
where $A_1,\dots,A_K$ are disjoint events with $\p(A_k)=2p$ for each $k$.
It is easy to check that $P_1,\dots,P_K$ are p*-variables,
and 
$$ \p(M(P_1,\dots,P_K)=q)=  \p\left(\bigcup_{k=1}^K A_k \right) = \sum_{k=1}^K \p(A_k)   = 2Kp.$$ 
Thus, $M(P_1,\dots,P_K) $ takes the value $ q  <Kp $  with probability $2Kp$, and it takes the value $1$ otherwise. 
Let $G$ be the left-quantile function of $M(P_1,\dots,P_K)$.
The above calculation leads to
$$\int_0^{2Kp}  G(t) \d t = 
2q Kp < \frac{(2Kp)^2} 2, 
$$
showing that $M(P_1,\dots,P_K)$ is not a p*-variable by \eqref{eq:quantile}, a contradiction. 

Next, assume $2Kp>1$. In this case, let $r= p-1/(2K)$, and define identically distributed random variables $P_1,\dots,P_K$ by
$$
P_k=  r  \id_{B_k} + p \id_{A_k}  + \id_{(A_k\cup B_k)^c},~~~k=1,\dots,K,
$$
where $A_1,\dots,A_K, B_1,\dots,B_K$ are disjoint events with $\p(A_k)=1/K-2r$ and $\p(B_k)=2r$, $k=1,\dots,K$. 
Note that the union of $A_1,\dots,A_K, B_1,\dots,B_K$  has probability $1$. 
It is easy to verify that  $P_1,\dots,P_K$ are p*-variables. 
Moreover, we have  
$q':=M(r,\dots,r) \le Kr$ since $M$ dominates $M_B$.
Hence,
$M(P_1,\dots,P_K) $ takes the value $ q' \le q  $  with probability $2Kr$, and it takes the value $q$ otherwise. Let $G$ be the left-quantile function of $M(P_1,\dots,P_K)$. Using $q'\le Kr$ and $q<Kp = Kr+1/2$, we obtain
\begin{align*} \int_0^{1}  G(t) \d t & = 
\int_0^{2Kr} q' \d t + \int_{2Kr}^1 q \d t  \\ &\le 2(Kr)^2 + q (1-2Kr) 
\\ &< 2(Kr)^2 +  \left(Kr+\frac12 \right) (1-2Kr)  = \frac 12,
\end{align*}
showing that $M(P_1,\dots,P_K)$ is not a p*-variable by \eqref{eq:quantile}, a contradiction.
As $M$ cannot strictly dominate $M_B$, we know that $M_B$ is admissible. 
\end{proof}

\subsection{Proofs of results in Section \ref{sec:5}}

\begin{proof}[Proof of Theorem \ref{prop:calibrator}] Let $U$ be a uniform random variable on $[0,1]$. 
\begin{enumerate}[(i)] 

\item 
The validity of the calibrator $u\mapsto (2u)\wedge 1$  is implied by (i), and below we show that it dominates all others.
For any function $g$ on $[0,\infty)$, suppose that $g(u) < 2u$ for some $u\in (0,1/2]$. 
Consider the random variable $V$ defined by $V=U1_{\{ U>2u\}} + u   1_{\{U\le 2u\}}$. Clearly, $V$ is a p*-variable. Note that
$$\p(g(V)\le g(u)) \ge \p(U\le 2u) =2u >g(u),$$  implying that $g$ is not a p*-to-p calibrator. Hence, any p*-to-p calibrator $g$ satisfies $g(u)\ge 2u$ for all $u\in (0,1/2]$, thus showing that $u\mapsto (2u)\wedge 1$ dominates all  p*-to-p calibrators. 
\item By \eqref{eq:quantile}, we know that $f(U)\ge_2 U$,  and thus $f$ is a valid p-to-p* calibrator. To show its admissibility, it suffices to notice that
$f$ is left-continuous (lower semi-continuous) function  on $[0,1]$,
and if $g\le f$ and $g\ne f$, then $\int_0^1 g(u) \d u<1/2$, implying that $g(U)$ cannot be a p*-variable.  \qedhere
\end{enumerate}
\end{proof}

\begin{proof}[Proof of Theorem \ref{prop:pstar-e}]
\begin{enumerate}[(i)]
\item Let $f$ be a convex  p-to-e calibrator.
Note that $-f$ is increasing and concave. 
For any $[0,1]$-valued p*-variable $P$, by definition, we have  
$\E[-f(P)] \ge \E[-f(U)].$ Hence, 
$$
\E[f(P)] = - \E[-f(P)] \le -\E[-f(U)] =\E[f(U)] \le1.
$$
Since a $[0,\infty)$-valued p*-variable is first-order stochastically larger than some $[0,1]$-valued p*-variable (e.g., \cite[Theorem 4.A.6]{SS07}),
we know  that $\E[f(P)] \le 1$ for all p*-variables $P$.
Thus, $f$ is a p*-to-e calibrator.

Next, we show the statement on admissibility. A convex admissible p-to-e calibrator $f$ is a p*-to-e calibrator.  Since the class of p-to-e calibrators is larger than the class of p*-to-e calibrators,  $f$ is not strictly dominated by any p*-to-e calibrator. 
\item 
We only need to show the ``only if" direction, since the ``if" direction is implied by (i).  
Suppose  that a non-convex function $f$ is an admissible p-to-e calibrator.
Since $f$ is not convex, there exist two points $t,s\in [0,1]$ such that 
$$f(t)+f(s) < 2 f\left(\frac{t+s}2\right).$$
Left-continuity of $f$ implies that there exists $\epsilon\in (0,|t-s|)$ such that 
$$f(t-\epsilon)+f(s-\epsilon) < 2 f\left(\frac{t+s}2\right).$$
Note that 
$$ \int_{t-\epsilon}^t   \left( f\left(\frac{t+s}2\right)  -f(u) \right)\d u \ge 
\epsilon\left( f\left(\frac{t+s}2\right) - f(t-\epsilon)\right),
$$
and the inequality also holds if the positions of $s$ and $t$ are flipped. 
Hence, by letting 
$A=  [t-\epsilon,t]\cup [s-\epsilon, t]$, we have 
\begin{align}  \notag 
& \int_A  \left( f\left(\frac{t+s}2\right)  -f(u) \right)\d u  
\\& \ge \epsilon\left(2 f\left(\frac{t+s}2\right)  -f(t-\epsilon)   -f(s-\epsilon) \right)  >0. \label{eq:convex}
\end{align} 
Let $U$ be a uniform random variable on $[0,1]$ and $P$ be given by
$$
P=U \id_{\{U\not \in A\}} + \frac{t+s}2 \id_{\{U\in A\}}.
$$
For any increasing concave function $g$ and   $x\in [t-\epsilon,t]$ and $y\in [s-\epsilon,s]$, we have 
$$2 g\left(\frac{t+s}2\right) \ge g(t) +g(s) \ge g(x) + g(y).$$
Therefore, $\E[g(P)] \ge \E[ g(U)]$, and hence $U\le_2 P$.
Thus, $P$ is a p*-variable.
Moreover, using \eqref{eq:convex}, we have 
$$
\E[f(P)]= \int_0^1 f(u) \d u + \int_A  \left( f\left(\frac{t+s}2\right)  -f(u) \right)\d u >\int_0^1 f(u)\d u=1.
$$
Hence, $f$ is not a p*-to-e calibrator. 
Thus, $f$ has to be convex if it is both an admissible p-to-e calibrator and a p*-to-e calibrator. 

\item
First, we show that $f:e\mapsto  (2e)^{-1}\wedge 1$ is an e-to-p* calibrator. 
Clearly, it suffices to show that $1/(2E)$ is a p*-variable for any e-variable $E$ with mean $1$, since any e-variable with mean less than $1$ is dominated by an e-variable with mean $1$. 
Let $\delta_x$ be the point-mass at $x$. 

Assume that $E$ has a two-point distribution (including the point-mass $\delta_1$ as a special case). With $\E[E]=1$, the distribution $F_E$ of $E$ can be characterized with two parameters $p\in (0,1) $ and $a\in (0,1/p]$ via
$$
F_E= p \delta_{1+(1-p)a}  + (1-p)  \delta_{1-pa}.
$$ 
The distribution  $F_P$ of $P:=1/(2E)$ (we allow $P$ to take the value $\infty$ in case $a=1/p$) is given by
$$F_P= p \delta_{1/(2+2(1-p)a)}  + (1-p)  \delta_{1/(2-2pa)}.$$
Let $G_P$ be the left-quantile function of $P$ on $(0,1)$.  
We have 
$$G_P(t) =  \frac{1}{2+2(1-p)a} \id_{\{t\in (0,p]\}} +  \frac{1}{2+2pa} \id_{\{t\in(p,1)\}}. 
$$
Define two functions $g$ and $h$ on $[0,1]$ by $g(v):=\int_0^v G_P(u) \d u  $ and $h(v):=v^2/2$. 
For $v\in (0, p]$, we have, using $a\le 1/p$,
\begin{align*}
g(v) &= \int_0^v G_P(u)\d u \\ &=  \frac{v}{2+2(1-p)a} \ge \frac{v}{2+2(1-p)/p} = \frac{vp}{2} \ge\frac{v^2}2 = h(v).
\end{align*}
Moreover, Jensen's inequality gives
$$g(1)=\int_0^1 G_P(u)\d u =\E[P]=\E\left[\frac{1}{2E}\right] \ge \frac 1{2\E[E]} =\frac 12 = h(1).$$
Since $g$ is linear on $[p,1]$, and $v$ is convex,   $g(p)\le h(p)$ and $g(1)\le h(1)$ imply $g(v)\le h(v)$ for all $v\in [p,1]$. 
Therefore, we conclude that $g\le h$ on $[0,1]$, namely
$$
\int_0^v G_P(u)\d u  \le \frac{v^2}2 ~~\mbox{for all $v\in [0,1]$}.
$$
Using \eqref{eq:quantile}, we have that $P$ is a p*-variable. 

For a general e-variable $E$ with mean $1$, its distribution can be rewritten as a mixture of two-point distributions with mean $1$ %(one needs to pair up  values taken by $E$). %
(see e.g., the construction in Lemma 2.1 of \cite{WW15}).
Since the set of p*-variables is closed under distribution mixtures (Proposition \ref{prop:convex}),
we know that $f(E)$ is a p*-variable. Hence, $f$ is an e-to-p* calibrator.

To show that $f$ essentially dominates all other e-to-p* calibrators, we take any e-to-p* calibrator $f'$. 
Using Theorem \ref{prop:calibrator}, the function
$e\mapsto (2 f' (e)) \wedge 1$ is an e-to-p calibrator. Using Proposition 2.2 of \cite{VW21}, 
any e-to-p calibrator is dominated by $e\mapsto e^{-1} \wedge 1$, and hence
$$
(2 f' (e)) \wedge 1 \ge e^{-1} \wedge 1~~~~\mbox{for $e\in [0,\infty)$},
$$
which in term gives $f'(e)\ge  (2e)^{-1}$ for $e\ge 1$.
Since $f'$ is decreasing, we know that $f'(e)<1/2$ implies $e>1$.
For any $e\ge 0$ with $f'(e)<1/2$, we have $f'(e)\ge  f(e)$, and thus $f$ essentially dominates $f'$. \qedhere  
\end{enumerate}
\end{proof}

\begin{proof}[Proof of Proposition \ref{prop:1over2e}]
The first statement is implied by Theorem \ref{prop:pstar-e} (iii)  .
For the second statement, we note that $1/2$ is a p*-variable.  For $\alpha \in (0,1)$, applying \eqref{eq:randomav} with $P=1/2$ and $V= \alpha  E $, we obtained a p*-test with a random threshold with mean at most $\alpha$. Using  Proposition \ref{prop:storder}, this test has size at most $\alpha$, that is,
$\p(2E \ge 1/\alpha)\le \alpha$.  Hence, $P$ is a p-variable.
\end{proof} 

  \section{Randomized p*-test and applications}
  \label{sec:testing}
  In this appendix we discuss several applications of testing with p*-values and randomization.
   \subsection{Randomized p*-test}

We first introduce the a generalized version of  the randomized p*-test  in Proposition \ref{th:newdef}. 
The following density condition (DP)  for a $(0,1)$-valued random variable $V$ will be useful.
  \begin{enumerate}[(DP)]  
  \item $V$ has a decreasing  density function on $(0,1)$.
  \end{enumerate}
 
The canonical choice of $V$ satisfying (DP) is a uniform random variable on $[0,2\alpha]$ for $\alpha \in (0,1/2]$, which we will explain later. 
For a $(0,1)$-valued random variable $V$ with mean $\alpha$ and a p*-variable $P$ independent of $V$, we consider the test
  \begin{align}\label{eq:randomav}
 \mbox{rejecting the null hypothesis} ~ \Longleftrightarrow~ P\le V.
   \end{align}    
   The following theorem justifies the validity of  the test \eqref{eq:randomav} with the necessary and sufficient condition (DP).
%   For any $(0,1)$-valued random variable $V$, if $P$ is uniformly distributed on $[0,1]$ and independent of $V$, then $\p(P\le V)=\E[V]$. Hence, the test \eqref{eq:randomav} has size at most $\alpha$ for all  p-variables $P$ if and only 
%      if $\E[V]\le \alpha$. We will see below that to guarantee the desired size for all p*-variables, the extra condition (DP) is both sufficient and necessary. 

   \begin{proposition}\label{th:1}
Suppose that $V$ is a $(0,1)$-valued random variable with mean $\alpha$. 
   \begin{enumerate}[(i)]
   \item The test \eqref{eq:randomav} has size at most $\alpha$ for all p*-variables $P$ independent of $V$ if and only if $V$ satisfies (DP). 
   \item The test \eqref{eq:randomav} has size at most $\alpha$  for all p-variables $P$ independent of $V$, and the size is precisely $\alpha$ if $P$ is uniformly distributed on $[0,1]$.
   \end{enumerate}
 \end{proposition} 
  
  The proof of Proposition \ref{th:1} relies on the following lemma. 
  \begin{lemma}\label{lem:1}
For any non-negative random variable $V$ with a decreasing density function on $(0,\infty)$ (with possibly a probability mass at $0$) and any p*-variable $P$   independent of $V$, we have $\p(  P\le V) \le \E[V]$.     \end{lemma}

 \begin{proof}%[Proof of Lemma \ref{lem:1}]
  Let $F_V$ be the distribution function of $V$, which   is an increasing concave function on $[0,\infty)$ because of the decreasing density.
  Since $P$ is a p*-variable, 
  we have   $\E[F_V( P)] \ge \int_0^1  F_V(u) \d u  $. 
Therefore,  
  \begin{align*}
 \p( P \le V)  &=  \E[\p(P\le V|P)]  
  =\E[1 -    F_V(  P)]   \le \int_0^1(1-F_V(u))\d u =\E[V].
  \end{align*}  
 Hence, the   statement  in the lemma holds. 
% If $P $ is  uniformly distributed on $[0,1]$, then $\p( P\le V) = \int_0^1 v \d F_V(v) =\alpha$. 
  \end{proof}
  \begin{proof}[Proof of Proposition \ref{th:1}]
  We first note that (ii) is straightforward:  for a uniform random variable $U$ on $[0,1]$   independent of $V$, then $\p(U\le V)=\E[V]$.
  If $P \ge_1 U$, then $\p(P\le V)\le\p(U\le V)\le \E[V]$.  
 
  The ``if" statement of point (i) directly follows from Lemma \ref{lem:1},  noting that condition (DP) is stronger than the  condition in Lemma \ref{lem:1}. Below, we show the ``only if" statement of point (i). 
  
%\begin{proposition}\label{th:2} 
%Suppose that  a $(0,1)$-valued random variable $V$  has mean $\alpha$ 
%and
%  $\p(  P\le V) \le \alpha$ for  arbitrary p*-variables $P$ independent of $V$.
%  Then $V$ satisfies (DP).
%\end{proposition} 
%\begin{proof} 
Let $F_V$ be the distribution function of $V$ and $U$ be a uniform random variable on $[0,1]$.
Suppose that   $F_V$ is not concave on $(0,1)$. 
It follows that there exists $x,y\in (0,1)$ such that 
$F_V(x)+F_V(y) > 2F_V((x+y)/2)$.
By the right-continuity of $F_V$, there exists $\epsilon\in (0,|x-y|)$ such that
\begin{align}
\label{eq:densitynecessary1}
F_V(x )+F_V(y ) > 2F_V\left(\frac{x+y+\epsilon} 2\right).
\end{align}
 Let $A [x,x+\epsilon] $ and $B = [y,y+\epsilon] $,
which are disjoint intervals.  
Define  a random variable $P$ by 
\begin{align}
\label{eq:densitynecessary}
P= U \id_{\{U\not \in A\cup B\}}   + \frac{x+y+\epsilon}2\id_{\{U\in A\cup B\}}.
\end{align}
We check that $P$ defined by \eqref{eq:densitynecessary} is a p*-variable. For any concave function $g$, Jensen's inequality gives
\begin{align*} \E[g(P)] & =\int_{[0,1]\setminus (A\cup B)}  g(u) \d u  + 2 \epsilon g\left( \frac{x+y+ \epsilon}2\right)
\\& \ge  \int_{[0,1]\setminus (A\cup B)}  g(u) \d u  +  \int_{A\cup B} g(u)\d u = \E[g(U)].
\end{align*}
Hence, $P$ is a p*-variable. It follows from \eqref{eq:densitynecessary1} and \eqref{eq:densitynecessary} that 
\begin{align*}
  \E\left[ F_V\left(   P\right)\right ] 
 & = \int_{[0,1]\setminus (A\cup B)}  F_V(u) \d u + \int_{A\cup B} F_V\left(\frac{x+y+\epsilon}2 \right) \d u
    \\ & < \int_{[0,1]\setminus (A\cup B)}  F_V(u) \d u + \int_{A\cup B} \frac{F_V(x)   + F_V(y) }2   \d u 
    \\ & = \int_{[0,1]\setminus (A\cup B)}  F_V(u) \d u +   \epsilon (F_V(x)   + F_V(y)) 
      \\ & = \int_{[0,1]\setminus (A\cup B)}  F_V(u) \d u +   \int_A F_V(x)  \d u   + \int_B  F_V (y) \d u 
      \\& \le \int_0^1 F_V(u)\d u =1-\E[V]=1-\alpha.
    \end{align*}
    Therefore, 
      \begin{align*}
 \p\left(  P \le V\right)    &=1 -     \E\left[ F_V\left(   P \right)\right ]  
 >1-(1-\alpha)=\alpha.
  \end{align*} 
  Since this contracts the validity requirement, we know that $V$ has to have a concave distribution function, and hence a decreasing density on $(0,1)$.  
\end{proof}

Lemma \ref{lem:1} gives $\p(P\le V)\le \E[V]$ for $V$ possibly taking values larger than $1$ and possibly having a probability mass at $0$.
  We are not interested designing a random threshold with positive probability to be $0$ or larger than $1$, but this result will become  helpful in Section \ref{sec:72}. 
    Since condition (DP) implies $\E[V]\le 1/2$, we will assume $\alpha \in (0,1/2]$, which is certainly harmless for practice.

With the help of  Proposition \ref{th:1}, we formally define $\alpha$-random thresholds and the randomized p*-test. 
\begin{definition}\label{def:V}
 For a significance level $\alpha \in(0,1/2]$,  an \emph{$\alpha$-random threshold} $V$ is a $(0,1)$-valued random variable  independent of the test statistics (a p*-variable $P $ in this section) with mean $\alpha$ satisfying (DP).
For an $\alpha$-random threshold $V$ and a p*-variable $P$, the \emph{randomized p*-test} is given by \eqref{eq:randomav}, i.e.,   rejecting the null hypothesis  $ \Longleftrightarrow  P\le V$.
  \end{definition} 

Proposition \ref{th:1} implies that the randomized p*-test always has size at most $\alpha$,  just like the classic p-test \eqref{eq:p-test}. 
   Since the randomized p*-test \eqref{eq:randomav} has size equal to $\alpha$ if $P$ is uniformly distributed on $[0,1]$,   the size $\alpha$ of the randomized p*-test cannot be   improved in general.

As mentioned in Section \ref{sec:pstar-test}, randomization is generally undesirable in testing. Like any other randomized methods, different scientists may arrive at different statistical conclusions by the randomized p*-test for the same data set generating the  p*-value.
Because of assumption (DP), which is necessary for validity by Proposition \ref{th:1}, we cannot reduce the $\alpha$-random threshold $V$ to a deterministic  $\alpha$. 
  This undesirable feature is the price one has to pay when a p-variable is weakened to a p*-variable.

If one needs to test with a deterministic threshold,
  then  $\alpha/2$ needs to be used instead of $\alpha$. In other words,  the test 
    \begin{align}\label{eq:halfalpha}
 \mbox{rejecting the null hypothesis} ~ \Longleftrightarrow~ P\le \alpha/2
   \end{align}    
   has size $\alpha$ for all p*-variable $P$. The validity of \eqref{eq:halfalpha} was noted by \citet{R82}, and it is a direct consequence of Proposition \ref{prop:trivial}. 
Unlike the random threshold $\mathrm U[0,2\alpha]$ which gives a size precisely $\alpha$ in realistic situations, the deterministic threshold $\alpha/2$ is often overly conservative in practice (see discussions in \cite[Section 5]{M94}), but it  cannot be improved in general when testing with the average of p-variables (\cite[Proposition 3]{VW20}); recall that the average of p-variables is a p*-variable.
  
We will see  an   application of the randomized p*-test  in Section \ref{sec:71},  leading to new tests on the weighted average of p-values,  which can be made deterministic if one of the p-values is independent of the others. 
Moreover,   the randomized p*-test can be used to improve the power of tests with   e-values  and  martingales  in Section \ref{sec:72}.

As  mentioned above, the extra randomness introduced by the random threshold $V$ is often considered undesirable.
One may wish to choose $V$ such that the randomness is minimized. 
The next result shows that $V\sim\mathrm U[0,2\alpha]$ is the optimal choice if the randomness is measured by variance or convex order.
\begin{proposition}\label{prop:con}
For any $\alpha$-random threshold $V$, we have $\var(V) \ge \alpha^2/3$, and this smallest variance is attained by $V^*\sim\mathrm U[0,2\alpha]$.
Moreover, $\E[f(V)] \ge \E[f(V^*)]$ for any convex function $f$ (hence $V\le_2 V^*$ holds). 
\end{proposition}
\begin{proof}%[Proof of Proposition \ref{prop:con}]
We directly show $\E[f(V)] \ge \E[f(V^*)]$ for all convex functions $f$, which implies the statement on variance as a special case since the mean of $V$ is fixed as $\alpha$. 
Note that $V$ has a concave distribution function $F_V$ on $[0,1]$, and $V^*$ has a linear distribution function $F_{V^*}$ on $[0,2\alpha]$. Moreover, they have the same mean. Hence, there exists $t\in [0,2\alpha]$ such that $F_V(x)\ge F_{V^*}(x) $ for $x\le t$ 
and $F_V(x)\le F_{V^*}(x) $ for $x\ge t$. 
This condition is sufficient for $\E[f(V)] \ge \E[f(V^*)]$ by Theorem 3.A.44 of \cite{SS07}.
\end{proof}

Combining Propositions \ref{th:1} and  \ref{prop:con}, the canonical choice of the  threshold in the randomized p*-test has a uniform distribution on $[0,2\alpha]$.

We note that it is also possible to use some $V$ with mean less than $\alpha$ and  variance less than $\alpha^2/3$.  This reflects a tradeoff between power and variance. Such a random threshold does not necessarily have a decreasing density. For instance, the point-mass at $\alpha/2$ is a valid choice; the next proposition gives some other choices.
 \begin{proposition}\label{prop:storder}
 Let $V$ be an $\alpha$-random threshold and $V'$ is a random variable satisfying $V'\le_1 V$. We have $\p(P\le V') \le \alpha$ for arbitrary p*-variable  $P$ independent of $V'$.
 \end{proposition} 
 \begin{proof}%[Proof of Proposition \ref{prop:storder}]
Let $F$ be the distribution function of $\bar P$. 
For any increasing function $f$, we have $\E[f(V')]\le \E[f(V)]$, which follows from $V' \le _1 V$.
Hence, we have
  \begin{align*}
  \p(  P\le V') = \E[F(V')] \le \E[F(V)] = \p(  P\le V)  \le \alpha,
  \end{align*}   
  where the last inequality follows from Proposition \ref{th:1}.
 \end{proof}

 Proposition \ref{prop:storder} can be applied to a special situation where  a p-variable $P$  and an independent p*-variable $P^*$ are available for  the same null hypothesis. 
 Note that in this case $2\alpha (1-P)\le_1 V \sim \mathrm U[0,2\alpha]$.
Hence,  by  Proposition \ref{prop:storder}, the test 
       \begin{align}\label{eq:ppstar-test} 
 \mbox{rejecting the null hypothesis}   ~\Longleftrightarrow~   \frac{P^*}{2(1-P)}   \le  \alpha.
     \end{align}  
 has size at most $\alpha$. Alternatively, using the fact that $\p(P\le 2\alpha (1-P^*))\le \alpha$ implied by Proposition \ref{th:1} (ii), 
 we can design a test 
        \begin{align}\label{eq:ppstar-test2} 
 \mbox{rejecting the null hypothesis}   ~\Longleftrightarrow~   \frac{P}{2(1-P^*)}   \le  \alpha.
     \end{align}  
The tests \eqref{eq:ppstar-test}  and \eqref{eq:ppstar-test2} both have a deterministic threshold of $\alpha$.
 This observation is useful in Section \ref{sec:71}.
  
     \subsection{Testing with averages of p-values}
\label{sec:71}
    In this section we illustrate  applications of the randomized p*-test to tests  with averages of   dependent p-values.
    
   Let $P_1,\dots,P_K$ be $K$ p-variables for a global  null hypothesis and they are generally not independent. 
  \citet{VW20} proposed testing using generalized means of the p-values, so that the type-I error is controlled at a level $\alpha$ under arbitrary dependence. We focus on the weighted (arithmetic) average $  \bar P:=\sum_{k=1}^K w_k P_k $   for  some weights $w_1,\dots,w_K\ge 0$ with $\sum_{k=1}^K w_k=1$. 
   In case $w_1=\dots=w_K=1/K$, we speak of the   arithmetic average. 
  
   %and all statements below are valid also for a weighted average of $(P_1,\dots,P_K)$. 
%    $$
% M_K(p_1,\dots,p_K)= 
%  \frac{1}{K}\sum_{k=1}^K p_k,~~~~p_1,\dots,p_K\in [0,1].
%  $$ 
  The method of \cite{VW20} on arithmetic average   is given by
  \begin{align}\label{eq:robustav}
 \mbox{rejecting the null hypothesis} ~ \Longleftrightarrow~ \bar P\le \alpha/2.
   \end{align} 
   We will call \eqref{eq:robustav} the \emph{arithmetic averaging test}. 
   The extra factor of $1/2$ is needed to compensate for arbitrary dependence among p-values.
Since $\bar P$ is a p*-variable by Theorem \ref{th:convex},  the test  \eqref{eq:robustav} is a special case of \eqref{eq:halfalpha}. 
   This method is quite conservative, and  it often has relatively low power compared to the Bonferroni correction and other similar methods unless p-values are very highly correlated, as illustrated by the numerical experiments in \cite{VW20}.

To enhance the power of the test \eqref{eq:robustav},  we apply the randomized p*-test  in Section \ref{sec:pstar-test} to design the  \emph{randomized averaging test} by
 \begin{align}\label{eq:randomav1}
 \mbox{rejecting the null hypothesis} ~ \Longleftrightarrow~ \bar P\le V,
   \end{align}     
   where $V$ is an $\alpha$-random threshold independent of $(P_1,\dots,P_K)$. 
  % Proposition \ref{th:1} justifies the size $\alpha$ of the randomized averaging test \eqref{eq:randomav1}
%for any weighted average of arbitrary p-variables. 
%  \begin{theorem}\label{th:1p}
%  The randomized p*-test \eqref{eq:randomav} has size at most $\alpha$; that is,  $\p(\bar P\le V) \le \alpha$ for arbitrary p-variables $P_1,\dots,P_K$ and $\alpha$-random threshold $V$.
%  If $P_1=\dots=P_K$ and each  is uniformly distributed on $[0,1]$, then $\p(\bar P\le V) = \alpha$.
%  \end{theorem}
%    
  Comparing the fixed-threshold test \eqref{eq:robustav} and the randomized averaging test  \eqref{eq:randomav1} with $V\sim \mathrm{U}[0,2\alpha]$,
 there is a $3/4$ probability that the randomized averaging test has a better power,  with the price of randomization.

Next, we consider a special situation, where 
 a p-variable among $P_1,\dots,P_K$ is independent of the others under the null hypothesis. 
 In this case, we can apply \eqref{eq:ppstar-test}, and 
 the resulting test is no longer randomized, as it is determined by the observed p-values.

%  From the proof of Proposition \ref{th:1}, one immediately finds that  the randomized averaging test \eqref{eq:randomav} is still valid if $\bar P$ is replaced by a weighted average, namely $\bar P=\sum_{k=1}^K w_k P_k$ for some $w_1,\dots,w_K\ge 0$ with $\sum_{k=1}^K w_k=1$.
% For the purpose of illustration, we focus our discussions on the symmetric arithmetic average.
% 

Without loss of generality, assume that $P_1$ is independent of $(P_2,\dots,P_K)$. 
Let $ \bar P_{(-1)}:=\sum_{k=2}^K w_k P_k$ be a weighted average of $(P_2,\dots,P_K)$. 
Using $ \bar P_{(-1)}$ as the p*-variable, the test    \eqref{eq:ppstar-test}  becomes 
  \begin{align}\label{eq:randomav-ind} 
 \mbox{rejecting the null hypothesis}   ~\Longleftrightarrow~   \bar P_{(-1)} + 2\alpha P_1 \le 2\alpha. 
   \end{align}
Following directly from the validity of \eqref{eq:ppstar-test}, %Proposition \ref{prop:storder}.   
for any p-variables $P_1,\dots,P_K$ with $P_1$  independent of $(P_2,\dots,P_K)$,  
the  test  \eqref{eq:randomav-ind} has size at most $\alpha$.
 
Comparing \eqref{eq:randomav-ind} with \eqref{eq:robustav}, we   rewrite \eqref{eq:randomav-ind} as 
  \begin{align*} 
 \mbox{rejecting the null hypothesis}   ~\Longleftrightarrow~   \bar P:= \sum_{k=1}^n w'_k P_k  \le \frac{2\alpha}{1+2\alpha},
     \end{align*} 
     where $$w'_1=\frac{2\alpha }{1+2\alpha} ~~\mbox{and}~~w'_k=  \frac{w_k}{1+2\alpha}, ~k=2,\dots,K.$$
     Note that  $ \bar P  $ is a weighted average of $(P_1,\dots,P_K)$.
     Since $\alpha$ is small, the rejection threshold is increased by almost three times, compared to  the test  \eqref{eq:robustav} applied to $ \sum_{k=1}^n w'_k P_k$ using the threshold $\alpha/2$.
For this reason, we will call   \eqref{eq:randomav-ind} the \emph{enhanced averaging test}. 
     
     In particular, if $2\alpha(K-1)=1$, and  $\bar P_{(-1)}$ is the arithmetic average of $(P_2,\dots,P_K)$, then   $\bar P 
  $ is the arithmetic average of $(P_1,\dots,P_K)$. 
For instance, if $K=51$ and $\alpha=0.01$, then the rejection condition for test \eqref{eq:randomav-ind} is $\bar P\le 1/51$, and the rejection condition for \eqref{eq:robustav} is    $\bar P\le 1/200$.
  
  \subsection{Simulation experiments}

We compare  by simulation the performance of  a few tests via merging p-values. 
For the purpose of illustration, we conduct correlated z-tests for the mean of normal samples with variance $1$. More pecisely, 
the null hypothesis   $H_0$ is $ \mathrm{N}(0,1)$ and the alternative is $ \mathrm{N}(\delta,1)$ for some $\delta>0$. 
 The p-variables $P_1,\dots,P_K$ are specified as $P_k=1-\Phi(X_k)$  from the Neyman-Pearson lemma, where $\Phi$ is the standard normal distribution function, and $X_1,\dots,X_K$ are generated from $\mathrm N(\delta,1)$ with pair-wise correlation $\rho$. As illustrated by the numerical studies in \cite{VW20},  the arithmetic average test performs poorly unless p-values are  strongly correlated. 
Therefore, we consider the cases where p-values are highly correlated,  e.g., parallel experiments with shared data or scientific objects. We set $\rho=0.9$ in our simulation studies; this choice is harmless as we are  interested in the relative performance of the averaging methods in this section, instead of their performance against other methods (such as   method of \citet{S86}) that are known work well for lightly correlated or independent p-values.

The significance level  $\alpha$ is set to be $ 0.01$. For a comparison, we consider the following tests:
\begin{enumerate}[(a)]
\item the arithmetic averaging test \eqref{eq:robustav}: reject $H_0$ if $ \bar P \le  \alpha/2$;
\item the randomized averaging test \eqref{eq:randomav1}: reject  $H_0$  if $ \bar P \le  V$ where $V\sim \mathrm{U}[0,2\alpha]$  independent of $\bar P$;
\item the Bonferroni method: reject $H_0$ if $ \min(P_1,\dots,P_K) \le  \alpha/K$;
\item the Simes method: reject $H_0$ if $ \min_k( K P_{(k)}/k) \le  \alpha $ where $P_{(k)}$ is the $k$-th smallest p-value;
\item the harmonic averaging test of \cite{VW20}: reject $H_0$ if $(\sum_{k=1}^K  P_k^{-1})^{-1} \le  \alpha/c_K $ where $c_K>1$ is a constant in \cite[Proposition 6]{VW20}.
\end{enumerate}
The validity (size no larger than $\alpha$) of the Simes method is guaranteed under some dependence conditions on the p-values; see \cite{S98,BY01}.  Moreover, as shown recently by \citet[Theorem 6]{VWW20}, 
the Simes method  dominates any symmetric and deterministic p-merging method valid for arbitrary dependence (such as the (a), (c) and (e);  the Simes method itself is not valid for arbitrary dependence).

In the second setting, we assume that one of the p-variables ($P_1$ without loss of generality) is independent of the rest, and the rest p-variables have a pair-wise correlation of $\rho=0.9$. For this setting, we further include  
\begin{enumerate}[(a)]
\item[(f)] the enhanced averaging test   \eqref{eq:randomav-ind}: reject $H_0$  if $   \bar P_{(-1)} + 2\alpha P_1 \le 2\alpha$.
\end{enumerate} 

The power (i.e., the probability of rejection) of each  test  
is computed from the average of 10,000 replications for varying signal strength $\delta$ and for $K \in\{20,100,500\}$. Results are reported in Figure \ref{fig:2a}. 

In the first setting of correlated p-values, the   randomized averaging test (b) improves the performance of (a) uniformly, at the price of randomization. The Bonferroni method (d) and the harmonic averaging test (e) perform poorly and are both penalized significantly as $K$ increases.  None of these methods visibly outperforms the Simes method, although in some situations the test (b) performs comparably to the Simes method.

In the second setting where an independent p-value exists, the enhanced arithmetic test (f) performs quite well;  it outperforms the Simes method for most parameter values especially for small signal strength $\delta$. This illustrates the significant improvement via incorporating an independent p-value. 

We remark that the averaging methods (a), (b) and (f) should not be used in situations in which correlation among p-values is known to be not very strong. This is because the arithmetic mean does not benefit from an increasing number of independent p-values of similar strength, unlike the methods of Bonferroni and Simes. 

\begin{figure}[t]
\begin{center}
 \makebox[\textwidth]{\includegraphics[width=4.2cm, trim=0 15 0 10, clip]{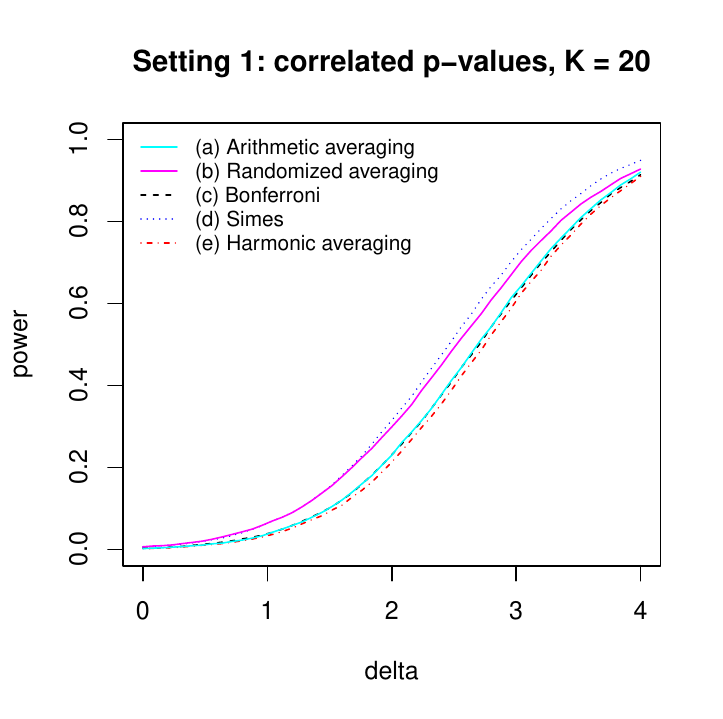}
 \includegraphics[width=4.2cm, trim=0 15 0 10, clip]{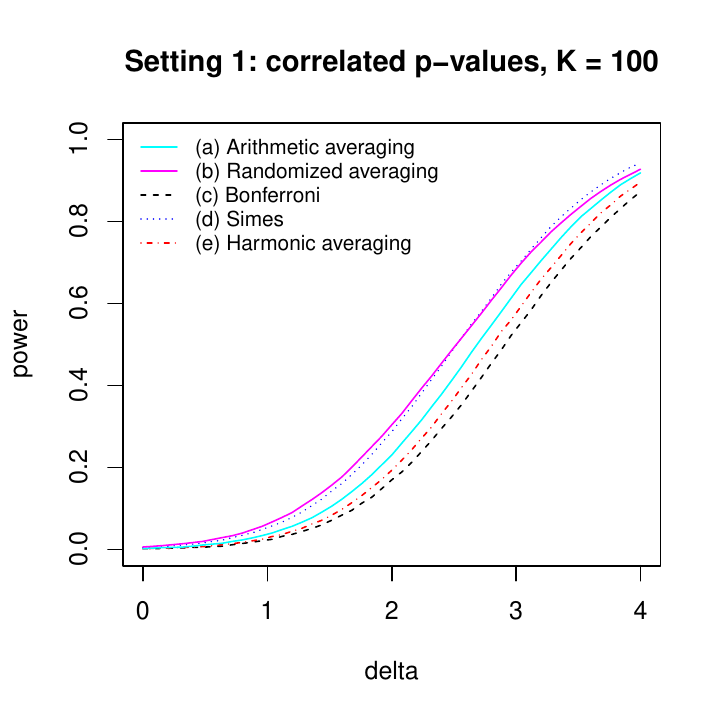}
  \includegraphics[width=4.2cm, trim=0 15 0 10, clip]{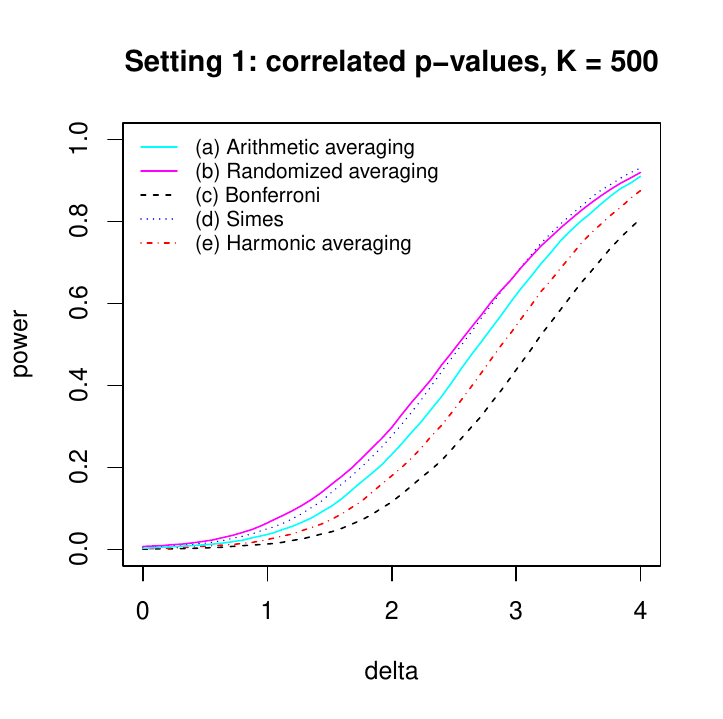}} \\
   \makebox[\textwidth]{\includegraphics[width=4.2cm, trim=0 15 0 10, clip]{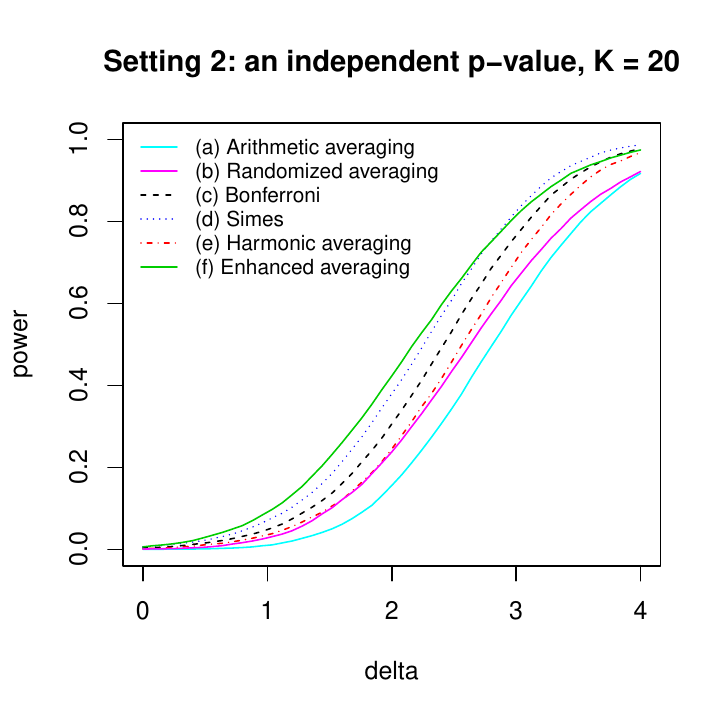}
 \includegraphics[width=4.2cm, trim=0 15 0 10,,clip]{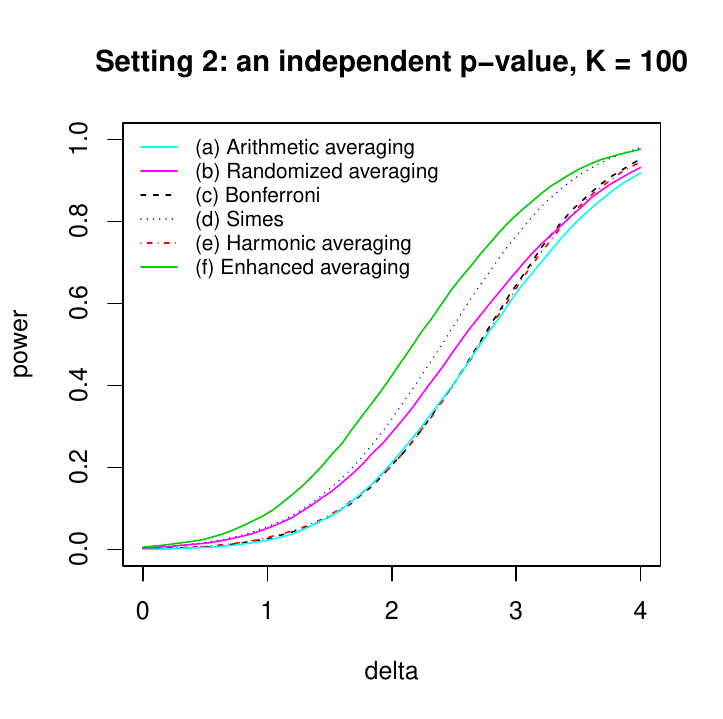}
  \includegraphics[width=4.2cm, trim=0 15 0 10, clip]{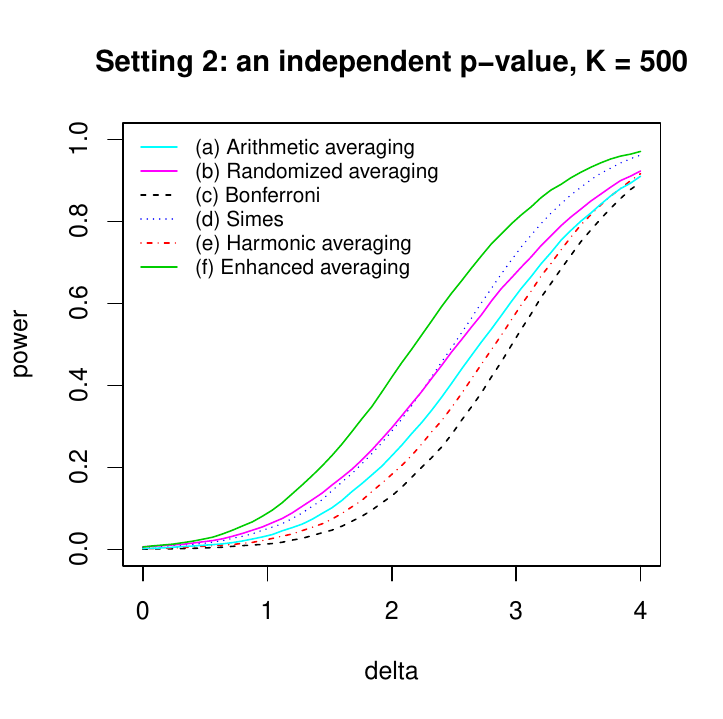}}
\caption{Tests based on combining p-values}
\label{fig:2a}
\end{center}
\end{figure}

\end{document}